\documentclass[11pt,a4paper]{amsart}
\usepackage{graphics}
\usepackage{amscd}
\usepackage{comment}
\usepackage{oldgerm}
\usepackage{amsmath}
\usepackage{amssymb}
\usepackage{amsthm}
\usepackage{color}

\usepackage[all]{xy}
\usepackage[foot]{amsaddr}

\makeatletter
\@addtoreset{equation}{section}
\makeatother

\theoremstyle{definition}
\newtheorem{defn}[equation]{Definition}
\theoremstyle{plain}
\newtheorem{thm}[equation]{Theorem}

\newtheorem{prop}[equation]{Proposition}
\newtheorem{fact}[equation]{Fact}
\newtheorem{cor}[equation]{Corollary}
\newtheorem{lem}[equation]{Lemma}

\theoremstyle{remark}
\newtheorem{rem}[equation]{Remark}
\newtheorem{ex}[equation]{Example}
\newcommand{\aaa}{\"a}

\newcommand{\Z}{\mathbb{Z}}
\newcommand{\N}{\mathbb{N}}
\newcommand{\R}{\mathbb{R}}
\newcommand{\C}{\mathbb{C}}
\newcommand{\T}{\mathbb{T}}

\newcommand{\del}{\partial}
\newcommand{\delb}{\overline{\partial}}

\begin{document}
\title[Spectral convergence in geometric quantization]{Spectral convergence in geometric quantization --- the case of toric symplectic manifolds}
\author[K. Hattori]{Kota Hattori}
\author[M. Yamashita]{Mayuko Yamashita}
\address[K. Hattori]{Keio University, 
3-14-1 Hiyoshi, Kohoku, Yokohama 223-8522, Japan}
\email{hattori@math.keio.ac.jp}
\address[M. Yamashita]{Research Institute for Mathematical Sciences, Kyoto University, 
606-8502, Kyoto, Japan}
\email{mayuko@kurims.kyoto-u.ac.jp}
\subjclass[]{}
\maketitle

\begin{abstract}
In this paper, we show the spectral convergence result of $\delb$-Laplacians when 
$(X,\omega)$ is a compact toric symplectic manifold equipped with the natural prequantum line bundle $L$. 
We consider a family $\{ J_s\}_s$ of $\omega$-compatible 
complex structures tending to the large complex structure limit, 
and obtain the spectral convergence of $\delb$-Laplacians 
acting on $L^k$. 
\end{abstract}

\tableofcontents

\section{Introduction}
This is the second paper of our project, where we analyze the limiting behavior of spectra of operators appearing in geometric quantization. 
Given a closed symplectic manifold $(X, \omega)$ and a prequantum line bundle $(L, \nabla, h)$ on it, we consider a one-parameter family of $\omega$-compatible complex structures $\{J_s\}_{s>0}$ which converges to a Lagrangian fibration $\mu : X \to B$ as $s \to 0$, in the sense of polarizations. 
Our goal is to show spectral convergence results for the family $\{\Delta^k_{\delb_{J_s}}\}_{s > 0}$ of $\delb$-Laplacians acting on sections of $L^k$. 
As a corollary, we expect to show that the family of quantum Hilbert spaces obtained by the K\aaa hler quantizations $\{J_s\}_{s>0}$ converges to that obtained by the real quantizations $\mu$. 
In our previous paper (\cite{HY2019}), we carried out this program in the case where the Lagrangian fibration $\mu$ is non-singular. 
In this paper, we show the corresponding convergence result for the case where $\mu$ is a moment map for a toric symplectic manifold.  

First we explain the motivation of our work. 
Given a symplectic manifold $(X, \omega)$, geometric quantization attempts to find nice representations of the Poisson algebra $C^\infty(X)$ on some Hilbert spaces called ``quantum" Hilbert spaces. 
Since we cannot expect to find a true representation on a Hilbert space which is ``small enough", we try to find a sequence of linear maps $\{C^\infty(X) \to \mathbb{B}(\mathcal{H}_k)\}_{k=1}^\infty$, called {\it strict deformation quantization},  that ``recovers" the Poisson algebra structure as $k \to \infty$. 
So one fundamental problem is to find a sequence of Hilbert spaces $\{\mathcal{H}_k\}_k$, which we also call quantum Hilbert spaces, and we focus on this aspect. 

Given a prequantized closed symplectic manifold $(X, \omega, L, \nabla, h)$, there are several known ways to construct quantum Hilbert spaces by choosing a polarization, an integrable Lagrangian subbundle of $TX \otimes \C$. 
A K\aaa hler polarization is given by choosing an $\omega$-compatible 
complex structure $J$ on $X = X_J$. 
In this case $\mathcal{H}_k = H^0(X_J,L^k)$, the space of holomorphic sections of $L^k$. 
On the other hand, a real polarization is given by choosing a Lagrangian fibration $\mu : X^{2n} \to B^n$. 
A point $b\in B$ is called a Bohr-Sommerfeld point of level $k$ if the space of pararell sections on $(L^k, \nabla^k)|_{\mu^{-1}(b)}$, denoted by $H^0(\mu^{-1}(b); (L^k, \nabla^k))$, is nontrivial. 
The set of Bohr-Sommerfeld points, $B_k \subset B$, is a discrete subset. 
In this case, the quantum Hilbert space is defined by $\mathcal{H}_k= \oplus_{b \in B_k}H^0(\mu^{-1}(b); (L^k, \nabla^k)\otimes \Lambda^{1/2}(\mu^{-1}(b)))$ (where $\Lambda^{1/2}(\mu^{-1}(b))$ is the vertical half form bundle). 

Since a real polarization can be regarded as a limit of K\aaa hler polarization $\{J_s\}_{s>0}$ as $s \to 0$, it is interesting to ask the quantum Hilbert spaces $H^0(X_{J_s}; L^k)$ also converges to the one obtained by the real polarization as $s \to 0$. 
This convergence is shown in the case of abelian varieties by Baier, Mour\~{a}o and Nunes in 
\cite{BMN2010}, and in the case of toric symplectic manifolds by Baier, Florentino, Mour\~{a}o and Nunes in \cite{BFMN2011}. 
Motivated by their works, we are interested in the following question: 
Since the space of holomorphic sections is the kernel of $\delb$-Laplacian $\Delta_{\delb_{J_s}}^k$ acting on $L^2(X_{J_s}; L^k)$, can we explain the convergence result of quantum Hilbert spaces from the viewpoint of spectral theory of $\delb$-Laplacians? 
More strongly, can we analyze the limiting behavior of the whole spectrum of $\delb$-Laplacians and relate them to real polarizations?
We gave an answer to this question in the previous work \cite{HY2019} in the case of non-singular Lagrangian fibrations, where we showed that the limit of the spectrum is the $\# B_k$-times direct sum of that of Harmonic oscillators. 
In this paper we give an answer in the toric case. 
We are able to show that similar limiting behavior also appears in this case. 

Now we explain the settings of this paper. 
Let $P \subset \R^n$ be a Delzant lattice polytope,  $(X_P, \omega)$ be the associated toric symplectic manifold and $\mu_P \colon X_P \to P$ be the moment map. 
The polytope $P$ also associates a prequantum line bundle $(L, \nabla, h)$ on $(X_P, \omega)$ in a canonical way. 
On $(X_P, \omega)$, we consider a family of $\omega$-compatible complex structures $\{J_s\}_{s>0}$ degenerating to the real polarization given by $\mu_P$, considered by Baier, Florentino, Mour\~{a}o and Nunes in \cite{BFMN2011} as follows. 
We consider a family of symplectic potentials of the form
\begin{align}\label{eq_symp_potential.y}
    v_P +\varphi + s^{-1}\psi, 
\end{align}
where $v_P \in C^\infty(P)$ is defined in \eqref{eq_v_P.y}, $\varphi \in C^\infty(P)$ satisfies some regularity condition explained in Section \ref{Ricci.h} and $\psi \in C_+^\infty(P)$ is a function with positive definite Hessian. 
Such a family of symplectic potentials determines a family of $\omega$-compatible complex structures $\{J_s\}_{s > 0}$ (\cite{Abreu1998}). 
As $s \to 0$,  the associated family of K\aaa hler polarizations on $TX_P \otimes \C$ converges to the real polarization given by $\mu_P$. 

The main result of this paper is the explicit description of the limit of spectrum of the $\delb$-Laplacians, $\Delta_{\delb_{J_s}}^k$, acting on $L^2(X_{J_s}; L^k)$ as $s \to 0$. 
To describe the limit, we prepare the following notations.
For a point $b \in P$, a cone $\mathcal{C}_b(\psi) \subset \R^n$ is defined, 
up to orthogonal transformations, 
by the equation
\begin{align*}
    \mathcal{C}_b(\psi) := (\mathrm{Hess}(\psi)_b)^{1/2}\mathcal{C}. 
\end{align*}
Here $\mathcal{C}$ is the cone in $\R^n$ which locally defines the polytope $P$ around $b$ (see Definition \ref{def_cone.y} for the precise definition). 

We denote the coordinate of $\R^n$ by $(\xi_1, \cdots, \xi_n)$, and denote by $\Delta^k_{\mathcal{C}_b(\psi)}$ the differential operator on $\mathcal{C}_b(\psi)$ defined by
\begin{align*}
\Delta_{\mathcal{C}_b(\psi)}^k
=\sum_{i=1}^n
\left( -\frac{\del^2 }{\del \xi_i^2} 
+2k \xi_i\frac{\del }{\del \xi_i}
\right),
\end{align*}
with the Neumann boundary condition. 
In Proposition 
\ref{prop_spec_Gaussian.y}, 
it is shown that this operator has 
compact resolvent on the weighted $L^2$ space $L^2(\mathcal{C}_b(\psi), e^{-k\|\xi\|^2}d\xi)$, and the multiplicity of the $0$-eigenvalue is one. 
Now our main result is the following. 

\begin{thm}\label{thm_main.y}
Let $(X_P,\omega)$ be a closed 
toric symplectic manifold of dimension $2n$ given by the 
Delzant polytope $P\subset \R^n$, $(L,\nabla,h)$ be the associated prequantum line bundle, 
$\mu \colon X_P \to B$ be the moment map 
and $k \ge 1$ be a positive integer.  
Let $\{J_s\}_{s > 0}$ 
be a family of compatible 
complex structures described in Section \ref{Ricci.h}. 
Then we have a compact spectral convergence, 

\begin{align*}
 (L^2(X_P,L^k),\Delta_{\delb_{J_s}}^k) \xrightarrow{s \to 0}\bigoplus_{b\in B_k} \left( L^2(\mathcal{C}_b(\psi), e^{-k\|\xi\|^2}d\xi),  \frac{1}{2}\Delta_{\mathcal{C}_b(\psi)}^k\right).
\end{align*}
in the sense of Kuwae-Shioya.
\end{thm}

Now we explain the strategy of the proof of Theorem \ref{thm_main.y}. 
The idea is similar to the one used in \cite{HY2019}. 
If we have a $\omega$-compatible complex structure $J$, it associates a Riemannian metric on $X$ defined by $g_J := \omega(\cdot, J\cdot)$. 
The metric $g_J$, together with the hermitian connection $\nabla$ on $L$, defines a Riemannian metric $\hat{g}_J$ on the frame bundle $S$ of $L$. 
We have a canonical isomorphism
\begin{align*}
    L^2(X, g_J; L^k)\simeq (L^2(S, \hat{g}_J)\otimes \C)^{\rho_k}, 
\end{align*}
where $\rho_k$ is the $S^1$ action given by principal $S^1$-action on $L^2(S, \hat{g}_J)$ and by the formula $e^{\sqrt{-1}t}\cdot z = e^{k\sqrt{-1}t}z$ on $\C$. 
Under this isomorphism, we have an identification of operators, 
\begin{align*}
    2\Delta_{\delb_J}^k = \Delta_{\hat{g}_J}^{\rho_k} - (k^2+nk), 
\end{align*}
where $\Delta_{\hat{g}_J}^{\rho_k}$ denotes the metric Laplacian on $(S, \hat{g}_J)$ restricted to the space $(L^2(S, \hat{g}_J)\otimes \C)^{\rho_k}$. 
In this way, we reduce the problem to the analysis of the spectral structure given by $((L^2(S, \hat{g}_J)\otimes \C)^{\rho_k}, \Delta_{\hat{g}_J}^{\rho_k})$. 
So the basic strategy is to consider the family $\{(S, \hat{g}_{J_s})\}_{s > 0}$ of Riemannian manifolds with isometric $S^1$-actions, analyze its Gromov-Hausdorff limit space and guarantee the spectral convergence to the operator on the limit space. 

The limit spaces are described in Section \ref{sec_lim_sp.y}. 
The main results there, Proposition \ref{prop_S^1-pmGH.h}, Proposition \ref{prop_lim_kBS.y} and Proposition \ref{prop_lim_nonBS.y} are summarized as follows. 
Let us take a point $b \in P$, and take any lift $u_b \in S$. 
The family of pointed metric measure spaces with the isometric $S^1$-action
\[
\{(S, \hat{g}_s, s^{-n/2}\nu_{\hat{g}_s}, u_b)\}_s
\]
converges to, as $s \to 0$,  
\begin{enumerate}
    \item $(\mathcal{C}_b(\psi) \times S^1, g_{l, \infty}, \det (\mathrm{Hess}(\psi)_b)^{-1/2}d\xi dt, (0, 1))$ if $b \in P \cap \frac{\Z^n}{l}$ and $b \notin P \cap \frac{\Z^n}{l'}$ for any integer $0 < l' < l$.  
    Here $g_{l,\infty}$ is the metric on $\R^n\times S^1$ defined by
\[
g_{l,\infty} := \frac{1}{l^2(1 +  \|\xi\|^2)}(dt)^2+ {}^t\!d\xi\cdot d\xi, 
\]
and the $S^1$-action on $\mathcal{C}_b(\psi) \times S^1$ is given by 
$e^{\sqrt{-1}\tau} \cdot (\xi, e^{\sqrt{-1}t}) = (\xi, e^{\sqrt{-1}(t + l\tau)})$.
    \label{intro_kBS.y}
    \item $(\mathcal{C}_b(\psi), {}^t\!d\xi \cdot d\xi, \det (\mathrm{Hess}(\psi)_b)^{-1/2}d\xi, 0)$ if $b \notin P \cap \frac{\Z^n}{l}$ for any integer $l$. 
    Here the $S^1$-action is trivial. 
    \label{intro_nonBS.y}
\end{enumerate}
in the sense of $S^1$-equivariant pointed measured Gromov-Hausdorff topology. 
The Laplacians on the limit spaces are described in Section \ref{sec_analysis_limit.y}. 
In the case \eqref{intro_nonBS.y}, since the $S^1$-action is trivial on the limit space (which we denote by $(S_\infty^b, g_\infty^b, \nu_\infty^b, p_\infty^b)$), we have $\left( 
L^2(S_\infty^b,\nu_\infty)
\otimes \C
\right)^{\rho_k} = \{0\}$ for any positive integer $k$. 
So in particular the limit Laplacian restricted to the $\rho_k$-equivariant subspace is trivial.  
In the case \eqref{intro_kBS.y}, we have $\left( 
L^2(S_\infty^b,\nu_\infty)
\otimes \C
\right)^{\rho_k} = \{0\}$ if $k \notin l\Z$, and if $k \in l\Z$, we have an isomorphism
\begin{align*}
L^2(\mathcal{C}_b(\psi),e^{-k\| \xi\|^2}d\xi)
\otimes\C
\cong 
\left( L^2 (S_\infty^b, \nu_\infty^b)
\otimes \C\right)^{\rho_{k}}, 
\end{align*}
and if we denote by $\Delta_\infty^{b, \rho_k}$ the limit Laplacian restricted to the $\rho_k$-equivariant subspace, we have the identification of operators
\begin{align*}
    \Delta_{\mathcal{C}_b(\psi)}^k \simeq \Delta_\infty^{b, \rho_k} -( k^2 + nk). 
\end{align*}
with the Neumann boundary condition. 
Thus, Theorem \ref{thm_main.y} is equivalent to the following Theorem \ref{thm_cpt_conv.h}.  

\begin{thm}\label{thm_cpt_conv.h}
Let $(X_P,\omega)$ be a closed 
toric symplectic manifold of dimension $2n$ given by the 
Delzant polytope $P\subset \R^n$, $(L,\nabla,h)$ be the associated prequantum line bundle, 
$\mu \colon X \to B$ be the moment map 
and $k \ge 1$ be a positive integer.  
Let $\{J_s\}_{s > 0}$ 
be a family of compatible 
complex structures corresponding to a family of symplectic potentials of the form \eqref{eq_symp_potential.y}. 
For each point $b \in P$, let 
\[
(S_\infty^b, g_\infty^b, \nu_\infty^b, p_\infty^b)
\]
be the pointed $S^1$-equivariant measured Gromov-Hausdorff limit space of the frame bundle $\{(S, \hat{g}_{J_s}, u_b)\}_{s > 0}$ as above. 
Put
\begin{align*}
H_s&=\left(
L^2\left(
S,s^{-n/2}\nu_{\hat{g}_{s}}
\right)\otimes \C
\right)^{\rho_k},\\
H_\infty&=
\bigoplus_{b\in B_k}
\left( 
L^2(S_\infty^b,\nu_\infty)
\otimes \C
\right)^{\rho_k}, 
\end{align*}
and consider the spectral structures $\Sigma_s$ and $\Sigma_\infty$ associated to the Laplacians restricted on $H_s$ and $H_\infty$, respectively. 
Then we have $\Sigma_s \to \Sigma_\infty$ compactly as $s \to 0$ 
in the sense of Kuwae-Shioya.
\end{thm}

To prove the desired compact spectral convergence, we have the following difficulties. 
\begin{enumerate}
    \item The Ricci curvatures of the family $\{(S, \hat{g}_s)\}_s$ have no uniform lower bound in general. 
    \label{intro_difficulty_1}
    \item The diameters of the family $\{(S, \hat{g}_s)\}_s$ are 
    unbounded, i.e., we have $\mathrm{diam}(S, \hat{g}_s) \to \infty$ as $s \to 0$. 
    \label{intro_difficulty_2}
\end{enumerate}
The difficulty \eqref{intro_difficulty_2} was also present in \cite{HY2019}, but the difficulty \eqref{intro_difficulty_1} is a new one here. 
The absence of lower bound for Ricci curvatures prevents us from using the well-developed theory for spectral convergence results of Laplacians directly. 
However, we are 
able to show spectral convergence in our situations, and much of the technical part of this paper is devoted to this point. 
On the other hand, the difficulty \eqref{intro_difficulty_2} is settled by the same method as in \cite{HY2019}, namely we have the localization results (Proposition \ref{thm_loc_bs_toric. y}) of $H^{1,2}$-bounded functions to Bohr-Sommerfeld fibers. 
This is proved by the same estimate as in \cite{HY2019}, coming from the idea which we called ``infinite dimensional Witten deformation". 

This paper is organized as follows. 
After recalling preliminary results from the theory of metric measure spaces in Section \ref{preliminaries.y}, we explain the settings of our problem in Section \ref{Ricci.h}. 
In Section \ref{sec_lim_sp.y}, we show the pointed $S^1$-equivariant measured Gromov-Hausdorff limits of the family $\{(S, \hat{g}_s, s^{-n/2}\nu_{\hat{g}_s}, u_b)\}_s$ as $s \to 0$. 
Based on this, in Section \ref{sec_analysis_limit.y}, we describe the Laplacians for the limit spaces. 
In Section \ref{sec_str_spec_conv.y}, we show the strong spectral convergence, which is a weaker notion of spectral convergence than compact spectral convergence, for our family. 
Finally in Section \ref{sec_cpt_conv.y}, we prove the compact convergence, which is our main theorem Theorem \ref{thm_main.y}. 

\noindent{\textbf{Notations. }}

\begin{itemize}
    \item For a Riemannian manifold $(M, g)$, let 
    \begin{align*}
    \nu_{g}&:=\mbox{the volume measure of }g,\\
    d_g&:=\mbox{the Riemannian distance of }g,\\
    B_g(p,r)&:=\{ x\in M;\, d_g(p,x)<r\}.
    \end{align*}
    \item For a positive definite matrix $A \in M_n(\R)$ and a nonnegative integer $0\le m \le n$, we denote
\begin{align}\label{eq_std_cone.y}
    C_m(A) := A^{1/2}(\R^m_{\ge 0} \times \R^{n-m}). 
\end{align}
\end{itemize}

\section{Preliminaries}\label{preliminaries.y}
In this section, we summarize preliminary notions and results needed in this paper. 
The contents in subsections \ref{subsec_spec_conv_general.y}, \ref{subsec_action_spec.h} and \ref{subsec_spec_conv_eq_lap.h} are essentially the same as those in \cite[Section 3]{HY2019}. 

\subsection{Convergence of spectral structures}\label{subsec_spec_conv_general.y}
In \cite{KuwaeShioya2003}, 
Kuwae and Shioya introduced the notion of 
spectral structures for the Laplacian which enabled us 
to treat the convergence of eigenvalues in the systematic way. 
In this subsection we review the framework developed 
in \cite{KuwaeShioya2003}. 
In this paper, Hilbert spaces are always assumed to be separable, and to be over $\mathbb{K}=\R$ 
or $\C$. 

Let $\mathcal{A}$ be a directed set, and fix an element $\infty \in \mathcal{A}$. 
The typical examples are $\mathcal{A} = \Z_{>0} \sqcup \{\infty\}$ and $\mathcal{A} = \R_{\ge 0}$ with $0 \in \R_{\ge 0}$ regarded as the element $\infty \in \mathcal{A}$. 

\begin{defn}
Let $\{H_\alpha\}_{\alpha \in \mathcal{A}}$ be a net of Hilbert spaces. 
The net $\{H_\alpha\}_\alpha$ is said to {\it converge to} $H_\infty$, 
if it is equipped with a dense subspace $\mathcal{C}\subset H_\infty$ 
and linear operators 
$\Phi_\alpha\colon \mathcal{C}\to H_\alpha$ 
which satisfy 
\begin{align}
\lim_{\alpha\to \infty}\| \Phi_\alpha(u)\|_{H_\alpha}=\| u\|_{H_\infty}
\label{conv str}
\end{align}
for any $u\in\mathcal{C}$. 
\label{def_conv_Hilb.h}
\end{defn}

\begin{defn}[{\cite[Definition 2.4 and 2.5]{KuwaeShioya2003}}]\label{def_conv_vector.y}
Let $\{H_\alpha\}_{\alpha \in \mathcal{A}}$ be a convergent net of Hilbert spaces and assume that $u_\alpha\in H_\alpha$ is given for each $\alpha\in\mathcal{A}$. 
\begin{itemize}
\setlength{\parskip}{0cm}
\setlength{\itemsep}{0cm}
 \item[(1)] A net $\{ u_\alpha\}_\alpha$ {\it converges to $u_\infty$ strongly} as $\alpha\to \infty$ 
if there exists a net $\{ \tilde{u}_\beta\}_{\beta\in \mathcal{B}} \subset H_\infty$ tending to 
$u_\infty$ such that 
\begin{align*}
\lim_{\beta}\limsup_{\alpha\to \infty}\| \Phi_\alpha(\tilde{u}_\beta) - u_\alpha\|_{H_\alpha} = 0.
\end{align*}
 \item[(2)] A net $\{ u_\alpha\}_\alpha$ {\it converges to $u_\infty$ weakly} as $\alpha\to \infty$ 
if \begin{align*}
\lim_{\alpha\to \infty}\langle u_\alpha,v_\alpha\rangle_{H_\alpha} = 
\langle u_\infty,v_\infty\rangle_{H_\infty}
\end{align*}
holds for any net $\{v_\alpha\}_{\alpha \in \mathcal{A}}$ such that 
$v_\alpha\to v_\infty$ strongly. 
\end{itemize}
\label{def spec}
\end{defn}

Next we define the notion of convergence of bounded operators. 
Suppose $\{H_\alpha\}_{\alpha \in \mathcal{A}}$ is a convergent net, and we have a net of bounded operators $\{B_\alpha \in L(H_\alpha)\}_{\alpha \in \mathcal{A} }$.
\begin{defn}[{\cite[Definition 2.6]{KuwaeShioya2003}}]\label{def_conv_op.y}
A net $\{B_\alpha\}_{\alpha\in \mathcal{A}}$ {\it strongly converges to} $B_\infty$ if $B_\alpha u_\alpha \to B_\infty u_\infty$ strongly for any net  
$\{u_\alpha\}_{\alpha \in \mathcal{A}}$ with $u_\alpha \in H_\alpha$ strongly converging to $u_\infty \in H_\infty$.  
$\{B_\alpha\}_{\alpha\in \mathcal{A}}$ {\it compactly converges to} $B_\infty$ if $B_\alpha u_\alpha \to B_\infty u_\infty$ strongly for any net $\{u_\alpha\}_{\alpha \in \mathcal{A}}$ with $u_\alpha \in H_\alpha$ weakly converging to $u_\infty \in H_\infty$.  

\end{defn}
Note that when $B_\alpha \to B_\infty$ compactly, $B_\infty$ is necessarily a compact operator. 

Next, we define the notion of spectral structure. 
\begin{defn}\label{def_specstr.y}
A {\it spectral structure} is a pair $(H, A)$, where $H$ is a Hilbert space and $A \colon \mathcal{D}(A) \to A$ is a densely defined self-adjoint linear operator on $H$. 
A spectral structure $(H, A)$ is {\it positive} if $A$ is a nonnegative operator. 
\end{defn}

\begin{rem}\label{rem_specstr.y}
\normalfont
The notion of spectral structure defined in Definition \ref{def_specstr.y} is more general than that in \cite[Section 2.6]{KuwaeShioya2003}; their definition corresponds to {\it positive} spectral structures in Definition \ref{def_specstr.y}. 
\end{rem}

For a spectral structure $(H_\alpha, A_\alpha)$ and a Borel subset 
$I \subset \mathbb{R}$, 
let $E_\alpha(I) \in B(H_\alpha)$ be the corresponding spectral projection of the self-adjoint operator $A_\alpha$ on $H_\alpha$. 
Put $n_\alpha(I) := \dim E_\alpha(I)H_\alpha \in \Z_{\ge 0} \cup \{\infty\}$. 

Now we define the convergence of spectral structures. 
In the below, when we consider a net of spectral structure $\{\Sigma_\alpha\}_\alpha = \{(H_\alpha, A_\alpha)\}_\alpha$, 
$\{H_\alpha\}_\alpha$ is supposed to be a converging net of Hilbert spaces. 
\begin{defn}[{\cite[Theorem 2.4 and Definition 2.14]{KuwaeShioya2003}}]\label{def_spec_conv.y}
Let $\{\Sigma_\alpha\}_{\alpha \in \mathcal{A}} =\{(H_\alpha, A_\alpha)\}_{\alpha \in \mathcal{A}}$ be a net of spectral structures.
We call that $\{\Sigma_\alpha\}_\alpha ${\it strongly} (resp. {\it compactly}){\it converges to} $\Sigma_\infty$ if $E_\alpha((\lambda, \mu]) \to E_\infty((\lambda, \mu])$ strongly (resp. compactly) for any real numbers $\lambda < \mu$ which are not in the point spectrum of $A_\infty$. 
\end{defn}

In terms of the spectrum of $A_\alpha$, the followings hold.  
\begin{fact}[{\cite[Proposition 2.6 and Remark 2.8]{KuwaeShioya2003}}]
Suppose that $a < b$ and both of them are not in the point spectrum of $A_\infty$. 
If $\Sigma_\alpha \to \Sigma_\infty$ strongly, then 
\[
\liminf_\alpha n_\alpha((a, b]) \ge n_\infty((a , b])
\]
holds.
\end{fact}

\begin{fact}[{\cite[Theorem 2.6 and Remark 2.8]{KuwaeShioya2003}}]
Assume that $\Sigma_\alpha \to \Sigma_\infty$ compactly. 
Then for any $a, b \in \R \setminus \sigma(A_\infty)$ with $a < b$, 
$n_\alpha((a, b]) = n_\infty((a , b])$ holds 
for $\alpha$ sufficiently close to $\infty$. 
In particular, the limit set of $\sigma(A_\alpha)$ coincides with $\sigma(A_\infty)$. 
\end{fact}

For a positive spectral structure $(H, A)$, 
its associated quadratic form $\mathcal{E} \colon H\to [0, \infty]$ 
is defined by $\mathcal{E}(u)
:=\|\sqrt{A}u\|^2_{H}$ 
for $u\in \mathcal{D}(\sqrt{A})$, and $\mathcal{E}(u):=\infty$ 
for $u\in H\backslash\mathcal{D}(\sqrt{A})$. 
Since $A$ is a closed operator, we see that $\mathcal{E}$ is 
{\it closed}, namely, 
$\mathcal{D}(\sqrt{A})$ is complete 
with respect to the norm 
defined by $\|u\|_{\mathcal{E}}:=\sqrt{\| u\|_H^2+\mathcal{E}(u)}$. 
We also have a notion of convergence for quadratic forms, as follows. 

\begin{defn}[{\cite[Definition 2.11 and 2.13]{KuwaeShioya2003}}]\label{def_quad_conv.y}
Let $\{H_\alpha\}_{\alpha \in \mathcal{A}}$ be a converging net of Hilbert spaces. 
A net of closed quadratic forms 
$\{ \mathcal{E}_\alpha\colon H_\alpha \to [0, \infty]\}_\alpha$ 
{\it Mosco converges to 
$\mathcal{E}_\infty\colon H_\infty \to [0, \infty]$ 
} as $\alpha\to \infty$ 
if 
\begin{itemize}
\setlength{\parskip}{0cm}
\setlength{\itemsep}{0cm}
 \item[(1)] $\mathcal{E}_\infty(u_\infty)
\le \liminf_{\alpha\to\infty}\mathcal{E}_\alpha(u_\alpha)$ for any $\{ u_\alpha\}_\alpha$
with $u_\alpha\to u_\infty$ weakly, and
 \item[(2)] for any $u_\infty\in H_\infty$ there exists $\{ u_\alpha\}_\alpha$ 
strongly converging to $u_\infty$ such that 
$\mathcal{E}_\infty(u_\infty) = \lim_{\alpha\to\infty}\mathcal{E}_\alpha(u_\alpha)$. 
\end{itemize}
Moreover, 
$\{ \mathcal{E}_\alpha\}_\alpha$ 
{\it converges to 
$\mathcal{E}_\infty$ compactly
} as $\alpha\to \infty$ 
if 
\begin{itemize}
\setlength{\parskip}{0cm}
\setlength{\itemsep}{0cm}
 \item[(3)] $\{ \mathcal{E}_\alpha\}_\alpha$ 
Mosco converges to 
$\mathcal{E}_\infty$ as $\alpha\to \infty$, and
 \item[(4)] for any $\{ u_\alpha\}_\alpha$ with 
$\limsup_{\alpha\to\infty}(\| u_\alpha\|_{H_\alpha}^2 + \mathcal{E}_\alpha(u_\alpha)) < \infty$, 
there exists a strongly convergent subnet. 
\end{itemize}
\label{def spec2}
\end{defn}

The spectral convergences of positive spectral structures have equivalent 
definitions in terms of convergence of associated quadratic forms, as follows. 

\begin{fact}[{\cite[Theorem 2.4]{KuwaeShioya2003}}]\label{fact_KS_conv.h}
Given a net of positive spectral structures $\{\Sigma_\alpha\}_\alpha = \{(H_\alpha, A_\alpha)\}_{\alpha}$ let us denote the corresponding net of quadratic forms by $\{\mathcal{E}_\alpha\}_\alpha$. 
Then the followings are equivalent. 
\begin{enumerate}
    \item We have a Mosco convergence $\mathcal{E}_\alpha \to \mathcal{E}_\infty$ (resp. $\mathcal{E}_\alpha \to \mathcal{E}_\infty$ compactly). 
    \item $\{\Sigma_\alpha\}_\alpha $ strongly (resp.  compactly) converges to $\Sigma_\infty$
\end{enumerate}
\end{fact}

Note that when $\mathcal{A} = \R_{\ge 0}$ with $0\in \R_{\ge 0}$ regarded as the limit element $\infty \in \mathcal{A}$, we see that any convergence of a net $\{X_s\}_{s > 0}$ is equivalent to the convergence of subsequence $\{X_{s_i}\}_{i \in \Z_{> 0}}$ for all $\{s_i\}_{i \in \Z_{> 0}}$ with $\lim_{i \to \infty} s_i = 0$. 
Thus in the below, we mainly work in the case where $\mathcal{A} = \Z_{> 0} \sqcup \{\infty\}$, i.e., we work with sequences.

\subsection{Lie group actions on Spectral structures}\label{subsec_action_spec.h}
Here, we explain the spectral structures 
induced by some spectral structures with 
compatible Lie group actions. 
See also \cite[Section 3.2]{HY2019}. 

Let $\Sigma$ be a spectral 
structure on $H$ whose infinitesimal 
generator is 
$A\colon \mathcal{D}(A)\to H$ 
and 
$G$ be a compact Lie group. 
Suppose that $G$ acts on 
$H$ linearly and isometrically, and $G\cdot \mathcal{D}(A)\subset \mathcal{D}(A)$ and suppose that 
$A$ is $G$-equivariant. 
For a finite dimensional 
unitary representation 
$(\rho, V)$ of $G$, 
put 
\begin{align*}
H^\rho &:=(H\otimes V)^\rho,\\
A^\rho &:=(A\otimes{\rm id}_V)|_{(\mathcal{D}(A)\otimes V)^\rho}
\colon (\mathcal{D}(A)\otimes V)^\rho\to H^\rho,
\end{align*}
then we have the spectral 
structure $\Sigma^\rho=(H^\rho,A^\rho)$. 

If $E$ and $E^\rho$ are the spectral measures of $A$, $A^\rho$, respectively, 
then 
\[
E^\rho((\lambda,\mu])
= E((\lambda,\mu])\otimes {\rm id}_V
\colon H^\rho \to H^\rho
\]
holds. 

Let $\{ \Sigma_\alpha=(H_\alpha,A_\alpha)\}_\alpha$ 
be a net of spectral structures 
and $\{ H_\alpha\}_\alpha$ converge to 
$H_\infty$.  
Let $\Phi_\alpha\colon \mathcal{C}
\to H_\alpha$ be as in Definition 
\ref{def_conv_Hilb.h} and assume that 
$G$ acts linearly and isometrically 
on all of $H_\alpha$, 
$A_\alpha$ are all $G$-equivariant, 
$G\cdot \mathcal{C}\subset \mathcal{C}$ 
and $\Phi_\alpha$ are $G$-equivariant. 
Put 
\begin{align*}
\mathcal{C}^\rho&:= (\mathcal{C}\otimes V)^\rho,\\
\Phi_\alpha^\rho&:= \Phi_\alpha
\otimes{\rm id}_V|_{\mathcal{C}^\rho}
\colon \mathcal{C}^\rho\to H_\alpha^\rho,\\
\end{align*}
then we can see that $\{ H_\alpha^\rho\}_\alpha$ 
converges to $H_\infty^\rho$ 
and the following proposition holds. 
\begin{prop}[{\cite[Proposition 3.11]{HY2019}}]
If $\Sigma_\alpha\to \Sigma_\infty$ 
strongly (resp.compactly), 
then $\Sigma_\alpha^\rho\to \Sigma_\infty^\rho$ 
strongly (resp.compactly).
\label{prop_spec_to_equiv_spec.h}
\end{prop}

\subsection{Laplacians on metric measure spaces}\label{subsec_lap.y}
In this subsection, we recall basic facts about Sobolev spaces on metric measure spaces. 
See \cite{AGS2014b}, \cite{gigli2015a} and \cite{gigli2015b} for more details. 
For a metric space $(X, d)$, we denote $B(x, r) := \{y \in X \ | \ d(x, y) < r\}$ for $x \in X$ and $r > 0$. 

\begin{defn}\label{def_mm_lap.y}
A {\it metric measure space} is a triple $(X, d, \nu)$ where $(X, d)$ is a separable metric space equipped with a Borel measure $\nu$ with $\mathrm{supp}(\nu) = X$. 
It is called {\it proper} if for all $x \in X$ and $r > 0$, we have $\nu(B(x, r)) < \infty$. 
\end{defn}
In this paper, we always assume that metric measure spaces we consider are proper. 

\begin{defn}
\begin{enumerate}
\setlength{\parskip}{0cm}
\setlength{\itemsep}{0cm}
 \item
The {\it Cheeger energy}  $\mathrm{Ch} : L^2(X, \nu) \to [0, +\infty]$ is a convex and $L^2(X, \nu)$-lower semicontinuous functional defined as follows. 
\begin{align*}
    \mathrm{Ch}(f) :=\inf \left\{\liminf_{n \to \infty}\frac{1}{2}\int_X (\mathrm{Lip}f_n)^2d\nu \ | \ f_n \in \mathrm{Lip}(X, d), \ \|f_n - f\|_{L^2}\to 0 \right\}, 
\end{align*}
where $\mathrm{Lip}f$ is the local Lipschitz constant of $f$. 
\item 
The Sobolev space $H^{1,2}(X, d, \nu)$ is defined as
\begin{align*}
    H^{1,2}(X, d, \nu) := \left\{f \in L^2(X, d)\ | \ \mathrm{Ch}(f) < +\infty \right\}. 
\end{align*}
We equip $H^{1,2}(X, d, \nu)$ with the norm
\begin{align*}
    \|f\|_{H^{1,2}} := \left( \|f\|^2_{L^2} + 2\mathrm{Ch} (f) \right)^{1/2}. 
\end{align*}
this space is a separable Hilbert space if $\mathrm{Ch}$ is a quadratic form (see \cite{AGS2014}). 
We say that $(X, d, \nu)$ is {\it infinitesimally Hilbertian} if $\mathrm{Ch}$ is a quadratic form. 
\item 
For an infinitesimally Hilbertian metric measure space $(X, d, m)$, we define its {\it Laplacian} $\Delta$ as the unbounded positive self-adjoint operator on $L^2(X, \nu)$ associated with the quadratic form $\mathrm{Ch}$. 
In other words, it is an unbounded positive operator characterized by the equation
\begin{align*}
    \mathrm{Ch}(f) = \langle f, \Delta f \rangle
\end{align*}
for all $f \in \mathcal{D}(\Delta)$. 
\end{enumerate}
\end{defn}

\begin{ex}
Let $(M, g)$ be a complete Riemannian manifold and $\phi \in C^\infty(M; \R)$ be a smooth function. 
Then we can consider the weighted Riemannian manifold $(M, d_g, e^{\phi}\nu_g)$ as a metric measure space. 
Then we have, for $f, g \in C_c^\infty(M)$, 
\begin{align*}
    \mathrm{Ch}(f) &= \frac{1}{2}\int_M\|df\|_g^2 e^{\phi} \nu_g, \\
    \|f\|_{H^{1,2}} &= \left( \|f\|^2_{L^2} +\int_M\|df\|_g^2 e^{\phi} \nu_g  \right)^{1/2}, \\
    \int_M f(\Delta h) e^\phi \nu_g &= \int_M \langle df, dh \rangle_g e^\phi \nu_g. \\
\end{align*}

\end{ex}

One of the important classes of metric measure spaces which we encounter in this paper is $\mathrm{RCD}(K, \infty)$ spaces for $K \in \R$. 
Although we do not give the definition of the $\mathrm{RCD}(K, \infty)$ condition here, we list some important properties as
follows. 

\begin{itemize}
\item For $D \in \R_{>0}$, $n \in \Z_{>0}$ and $K \in \R$, we denote by $\mathcal{M}(D, n, K)$ the set of closed Riemannian manifolds $(M, g)$ with $\mathrm{diam} (M, g) \le D$, $\dim M = n$ and $\mathrm{Ric}_g \ge Kg$. 
Let us denote the closure of $\mathcal{M}(D, n, K)$ by the measured Gromov-Hausdorff topology by $\bar{\mathcal{M}}(D, n, K)$. 
Then we have $\bar{\mathcal{M}}(D, n, K) \subset \mathrm{RCD}(K, \infty)$ for any $D$ and $n$. 
\item An $\mathrm{RCD}(K, \infty)$ space is infinitesimally Hilbertian. 
\item For an $\mathrm{RCD}(K, \infty)$ space $(X, d, \nu)$ with $\nu(X) < +\infty$, the embedding $H^{1,2}(X, d, \nu) \hookrightarrow L^2(X, \nu)$ is compact (\cite[Theorem 6.7]{GMS2015}). 
In particular, its Laplacian $\Delta$ has compact resolvent. 
\end{itemize}

\subsection{Strong spectral convergence of equivariant Laplacians}\label{subsec_spec_conv_eq_lap.h}
In this subsection, we explain how to apply the general theory of subsection \ref{subsec_spec_conv_general.y} to our situations. 

In this subsection we consider 
pointed metric measure 
spaces $(P_i,d_i,\nu_i,p_i)$ 
for $i\in\N\cup\{ \infty\}$, 
and we suppose that 
a compact Lie group $G$ acts on 
all of $(P_i,d_i,\nu_i,p_i)$ isometrically. 
In \cite{HY2019}, 
we defined the 
notion of the pointed 
$G$-equivariant measured Gromov-Hausdorff 
convergence denoted by 
\begin{align*}
(P_i,d_i,\nu_i,p_i)
\xrightarrow{G\mathchar`-\mathrm{pmGH}}
(P_\infty,d_\infty,
\nu_\infty,p_\infty), 
\end{align*}
as the special case of the convergence 
defined by Fukaya and Yamaguchi in 
\cite[Definition 4.1]{FukayaYamaguchi1994}. 
To define it, 
we take the Borel $G$-equivariant 
$\varepsilon_i$-approximation 
\begin{align*}
\phi_i\colon (\pi_i^{-1}(B(\bar{p}_i,R_i')),p_i)\to (\pi_\infty^{-1}(B(\bar{p}_\infty,R_i)),p_\infty),
\end{align*}
such that $\lim_{i\to\infty}\varepsilon_i=0$, 
$\lim_{i\to\infty}R_i=\lim_{i\to\infty}R_i'=\infty$, 
where $\pi_i\colon P_i\to P_i/G$ is the 
quotient map and $\bar{p}_i=\pi_i(p_i)$. 
Here, the metric on $P_i/G$ is 
defined by the distance between 
the $G$-orbits. 
See \cite[Definition 3.12]{HY2019} 
for the precise definition of 
the above convergences and 
approximation maps.

Here, for all $i\in\N$, 
we assume that 
$(P_i,g_i)$ are smooth Riemannian manifolds 
with isometric $G$-actions 
and $d_i,\nu_i$ are 
the Riemannian distances and Riemannian measures, 
respectively. 
Let $N$ be a positive integer and 
take points $p_i^j\in P_i$ for each 
$i \in \N$, $1 \le j \le N$, 
and assume that for each $j \neq l$, we have $\lim_{i \to \infty} d_i(p_i^j, p_i^l) = \infty$. 
We also assume that for each $1\le j\le N$, 
there is a pointed metric measure space $(P_\infty^j, d_\infty^j, \nu_\infty^j, p_\infty^j)$ with isometric measure-preserving 
$G$-action such that
\begin{equation}\label{eq_conv_hoge.y}
(P_i, d_i, \nu_i, p_i^j) \xrightarrow{G\mathchar`-\mathrm{pmGH}} (P_\infty^j, d_\infty^j, \nu_\infty^j, p_\infty^j). 
\end{equation}
We also assume that the limit spaces $(P_\infty^j, d_\infty^j, \nu_\infty^j, p_\infty^j)$ are $\mathrm{RCD}(K, \infty)$-spaces for some $K\in \R$, so that the Laplace operator $\Delta_{\infty}^j$ acting on $L^2(P_\infty^j, \nu_\infty^j)$ makes sense. 

Fix a positive integer $k \ge 1$ 
and put  
\begin{align*}
H_i &:= L^2(P_i, \nu_i) , \\
H_\infty &:= \bigoplus_{j = 1}^{N} L^2(P_\infty^j, \nu_\infty^j). 
\end{align*}
Then we obtain $H_i^\rho$ 
and $H_\infty^\rho$ in the same way 
as Subsection \ref{subsec_action_spec.h}. 
Now we explain the natural choice of $\mathcal{C}$ and $\Phi_i$. 
By (\ref{eq_conv_hoge.y}), we may take 
numbers $\epsilon_i, R'_i, R_i$ such that
$\lim_{i\to\infty}\varepsilon_i=0$ 
and $\lim_{i\to\infty}R_i = \lim_{i\to\infty}R'_i
=\infty$ and 
$G$-equivariant 
$\varepsilon_i$-approximation 
\begin{align*}
\phi_i^j\colon \pi_i^{-1}(B(\bar{p}^j_i,R'_i))\to 
\pi_\infty^{-1}(B(\bar{p}^j_\infty,R_i))
\end{align*}
such that $\phi_i(p_i) = p_\infty$. 
Moreover, by the assumption that
$\lim_{i \to \infty} d_i(p_i^j, p_i^l) = \infty$ for $j \neq l$, we may assume that
$\{\pi^{-1}_i(B(\bar{p}^j_i,R'_i))\}_{j=1}^N$ are mutually disjoint for each $i$.  
Thus we can set
\begin{align*}
\mathcal{C}:=\bigoplus_{j=1}^N C_c(P^j_\infty)
=\left\{ \sum_{j=1}^N f_j\in \bigoplus_{j = 1}^N C(P^j_\infty);\, {\rm supp}(f_j)\mbox{ is compact.}\right\},
\end{align*}
\[ 
\Phi_i(f)(u):=
\left\{ 
\begin{array}{cc}
f\circ\phi_i^j(u) & u\in \pi_i^{-1}(B(\bar{p}^j_i,R'_i))\\
0 & u\notin \pi_i^{-1}(B(\bar{p}^j_i,R'_i)) \mbox{ for any }j
\end{array}
\right.
\]
for $f\in \mathcal{C}$. 
Then the same procedure 
in Subsection \ref{subsec_action_spec.h} gives 
$\mathcal{C}^\rho$ and 
$\Phi_i^\rho$. 

Set
$A_i := \Delta_i$ and $A_\infty := \bigoplus_{j=1}^N \Delta_\infty^j$. 
Then we obtain $\Sigma_i^\rho$ 
and $\Sigma_\infty^\rho$ in 
the same way as Subsection \ref{subsec_action_spec.h}. 
Then we have the following results. 
\begin{fact}[{\cite[Proposition 3.16]{HY2019}}]\label{fact_str_conv_multiple.y}
Under the convergence 
\eqref{eq_conv_hoge.y}, 
assume moreover that
there exist $n \in \Z_{>0}$ and $\kappa >0$ such that for all $i \in \Z_{>0}$, we have
\[
\dim P_i = n \mbox{ and } \mathrm{Ric}(g_i) \ge \kappa g_i. 
\]
Then we have $\Sigma_i^\rho \to \Sigma_\infty^\rho$ strongly.
\end{fact}

\section{Settings}\label{Ricci.h}
Let $P\subset \R^n$ be a Delzant lattice polytope, 
which is given by 
\begin{align*}
P=\left\{ x\in \R^n;\, {}^t\nu_r x \ge \lambda_r,\, r=1,\ldots,d \right\} 
\end{align*}
for some $\nu_r\in\Z^n$ and 
$\lambda_r\in\Z$. 
For the Delzant construction 
of toric symplectic manifolds, 
see \cite{daSilva}.
Denote by $(X_P,\omega)$ the associated smooth 
toric symplectic manifold and by $\mu_P\colon X_P \to P$ the 
moment map. 
Let $\breve{P}$ be the interior of $P$, then 
the torus action on $\breve{P}$ is free and 
we have an identification $\breve{X}_P:=\mu^{-1}(\breve{P})=\breve{P}\times \mathbb{T}^n$ 
and the action-angle coordinate 
\begin{align*}
( x,\theta ) = ( x_1,\ldots,x_n,\theta^1,\ldots,\theta^n )
\end{align*}
on $\breve{X}_P$ such that 
$\omega=dx_i\wedge d\theta^i$ holds. Here, $x$ can be taken such that 
$x=\mu_P$. 

Put 
\begin{align}\label{eq_v_P.y}
v_P (x) &= \sum_{r=1}^d 
({}^t\nu_r \cdot x - \lambda_r ) \log 
({}^t\nu_r \cdot x - \lambda_r).
\end{align}
Let $C^\infty_{v_P}(P)\subset C^\infty(P)$ 
be the set consisting of 
functions $\varphi\in C^\infty(P)$ such that 
${\rm Hess}_x(v_P + \varphi)$ is positive definite 
on $\breve{P}$ and 
\begin{align*}
\det \left( {\rm Hess}_x(v_P + \varphi)\right)
\prod_{r=1}^d ({}^t\nu_r \cdot x - \lambda_r )
\end{align*}
is smooth and positive on $P$. 
Let $C^\infty_+(P)$ be the set consisting of 
the functions $\psi\in C^\infty(P)$ 
such that ${\rm Hess}_x(\psi)$ is positive definite 
on $P$. 
In this article we fix $\varphi\in C^\infty_{v_P}(P)$ 
and $\psi\in C^\infty_+(P)$ then put 
\begin{align*}
G_s:={\rm Hess}_x\left( v_P +\varphi + s^{-1}\psi\right)
= (G_{s,ij})_{i,j=1}^n
\end{align*}
for $s>0$ and define 
the complex structure $J_s$ on $X$ by 
\begin{align*}
J_s\left( \frac{\del}{\del x_i}\right) 
&= \sum_{j=1}^n G_{s,ij}\frac{\del}{\del \theta_j}\\
J_s\left( \frac{\del}{\del \theta_i}\right) 
&=
- \sum_{j=1}^n G_s^{ij}\frac{\del}{\del x_j}. 
\end{align*}
Here, $(G_s^{ij})_{i,j=1}^n = G_s^{-1}$.
Now, $\omega$ can be regarded as 
a K\"ahler metric on the 
complex manifold $(X,J_s)$ for a fixed $s$. 

Next we explain the notation appearing in the description of the limit spaces in Section \ref{sec_lim_sp.y}. 
\begin{defn}\label{def_cone.y}
Assume we are given a Delzant polytope $P \subset \R^n$ and a function $\Phi \in C^\infty_+(P)$ with positive definite Hessian. 
For a point $b \in P$,  
We define a cone $\mathcal{C}_b(\Phi) \subset \R^n$ as follows. 
Around the point $b$, $P$ locally coincides with the set $b + \mathcal{C}$
for some cone $\mathcal{C}$ in $\R^n$.
Using this we define
\begin{align*}
    \mathcal{C}_b(\Phi) := (\mathrm{Hess}(\Phi)_b)^{1/2}\mathcal{C}. 
\end{align*}
\end{defn}
Note that, up to orthogonal transformations on $\R^n$, the cone $\mathcal{C}_b(\Phi) $ is well-defined under affine coordinate change on $P$ (remark that the cone $\mathcal{C}$ and the matrix $\mathrm{Hess}(\Phi)_b$ depends on the choice of coordinate). 

\subsection{Geometric quantization}
In this subsection, we recall the general settings of geometric quantization. 

Let $(X,\omega)$ be a closed 
symplectic manifold of dimension $2n$ 
equipped with a 
prequantum line bundle 
$(L,\nabla,h)$, 
that is, 
$(\pi\colon L\to X,h)$ is a complex hermitian line bundle 
and $\nabla$ is a connection on 
$L$ preserving $h$ whose curvature form $F^\nabla$ 
is equal to $-\sqrt{-1}\omega$. 
Then $\nabla$ lifts 
to the connection form 
on the principal $S^1$-bundle 
\begin{align*}
S := S(L,h) := \{ u\in L;\, h(u,u)=1\},
\end{align*}
which gives the 
decomposition of $T_uS$ 
into the horizontal and 
vertical subspaces. 

An almost complex structure $J$ is 
called 
{\it $\omega$-compatible} if 
\begin{align*}
\omega(J\cdot,J\cdot)=\omega,\quad 
g_J:=\omega(\cdot,J\cdot) > 0.
\end{align*}
In \cite{HY2019}, 
we defined a Riemannian metric 
$\hat{g}_J$ on $S$ by 
\begin{align*}
\hat{g}_J := A\otimes A
+ (d\pi|_H)^* g_J,
\end{align*}
where $\sqrt{-1}A$ is a connection form on $S$ corresponding to 
$\nabla$ and $H\subset TS$ is 
the horizontal distribution. 
Note that the $S^1$-action 
on $S$ preserves $\hat{g}_J$. 

Denote by $\Gamma(L)$ the set of  $C^\infty$-sections of $L$ 
and let $L^k$ be 
the $k$-times tensor product of $L$. 
Then $L^k$ can be regarded as the associate 
bundle $L^k=S\times_{\rho_k} \C$, where 
$\rho_k$ is a $1$-dimensional unitary representation of $S^1$ 
defined by $\rho_k(\sigma)=\sigma^k$ for $\sigma\in S^1$. 
If $J$ is integrable, 
then $\nabla$ induces the 
holomorphic structure 
on $L^k\to X_J$ with 
the Chern connection 
$\nabla_{\delb{J}}$. 
Here $X_J$ is the complex manifold $(X,J)$.
Put 
\begin{align*}
\Delta_{\delb_J}^k := (\nabla_{\delb_J})^*\nabla_{\delb_J}
\colon \Gamma(L^k) \to \Gamma(L^k). 
\end{align*}
Under the natural identification 
\begin{align}
\Gamma(L^k) &\cong (C^\infty (S)\otimes \C )^{\rho_k}\label{associate.h}\\
&= \left\{ f\colon S \stackrel{C^\infty}{\rightarrow} \C ;\, 
\forall u\in S,\, \forall \sigma\in S^1,\,
\sigma^k f(u\sigma) = 
f(u) \right\},\notag
\end{align}
we have 
\begin{align}\label{eq_delbarlap.y}
2\Delta_{\delb_J}^k = \Delta_{\hat{g}_J}^{\rho_k} - (k^2+nk)
\end{align}
by \cite[Section 3]{hattori2019}, 
where 
$\Delta_{\hat{g}_J}^{\rho_k}$ 
is the Laplacian of 
$\hat{g}_J$ acting on 
$(C^\infty (S)
\otimes \C )^{\rho_k}$. 
In particular, the space of holomorphic sections 
$H^0(X_J,L^k)$ is identified with 
the $(k^2+nk)$-eigenspace of $\Delta_{\hat{g}_J}^{\rho_k}$.

A smooth map $\mu$ from $X$ to 
a smooth manifold $B$ of dimension $n$ is called a
Lagrangian fibration if 
$\mu$ is surjective 
and $\mu^{-1}(b)$ are Lagrangian submanifolds 
for all $b\in B$. 
By the definition of the prequantum line bundle, 
the restriction $L^k|_{\mu^{-1}(b)} \to \mu^{-1}(b)$ is 
a flat complex line bundle. 
\begin{defn}
$(1)$ For a Lagrangian fibration $\mu\colon X\to B$ 
with connected fibers, $\mu^{-1}(b)$ is a 
{\it Bohr-Sommerfeld fiber of level $k$} if 
$L^k|_{\mu^{-1}(b)} \to \mu^{-1}(b)$ has 
a nonzero flat section. 
$(2)$ $b\in B$ is a 
{\it Bohr-Sommerfeld point of level $k$} if $\mu^{-1}(b)$ is 
a Bohr-Sommerfeld fiber of level $k$. 
$(3)$ $b\in B$ is a 
{\it strict Bohr-Sommerfeld point of level $k$} if $b$ is 
a Bohr-Sommerfeld point of level $k$ 
and never be a Bohr-Sommerfeld point of level $k'$ for any $k'<k$. 
\end{defn}

\subsection{Ricci curvature}
In this subsection we compute the Ricci form of $(X, J_s)$ along 
\cite{Guillemin1994} and 
\cite{Abreu1998}. 

Since $(\C^{\times})^n$ acts 
freely and holomorphically on $\breve{X}_P$, 
there is a local holomorphic coordinate 
$z=(z_1,\ldots,z_n)$ on $\breve{X}_P$ 
such that 
\begin{align*}
\omega=2\sqrt{-1}\del\delb F_s
=\frac{\sqrt{-1}}{2}\sum_{i,j}\frac{\del^2 F_s}{\del \xi_i 
\del \xi_j} dz_i\wedge d\bar{z}_j
\end{align*}
for a function $F_s=F_s(\xi)$ on $\breve{X}_P$, 
where $\xi={\rm Re}(z)$. 
Here, $F_s$ is given as follows. 
The relation between $x$ and $\xi$ 
are given by 
\begin{align*}
x_i&=\frac{\del F_s}{\del \xi_i},\quad 
\xi_i=\frac{\del }{\del x_i}
\left( v_P +\varphi + s^{-1}\psi\right),
\end{align*}
then the matrix 
$(\frac{\del^2 F_s}{\del \xi_i \del \xi_j})$ is the inverse of 
$G_s$. 

The Ricci form $\rho_s$ is given by 
\begin{align*}
\rho_s = -\sqrt{-1}\del\delb 
\det\left( \frac{\del^2 F_s}{\del \xi_i \del \xi_j}\right)
= \sqrt{-1}\del\delb 
\det G_s.
\end{align*}
Since 
\begin{align*}
\frac{\del x_j}{\del \xi_i}=\frac{\del^2 F_s}{\del \xi_i\del \xi_j}=G_s^{ij},\quad 
\frac{\del \xi_j}{\del x_i}=G_{s,ij}
\end{align*}
hold, we have 
\begin{align}
\rho_s = \sqrt{-1}\del\delb \det G_s 
&= \frac{\sqrt{-1}}{4}\sum_{i,j}\frac{\del^2 (\det G_s)}{\del\xi_i\del\xi_j}
dz_i\wedge d\bar{z}_j \notag \\
&= \frac{\sqrt{-1}}{4} \sum_{i,j,k,l}G_s^{ik}\frac{\del}{\del x_k}\left( G_s^{jl}\frac{\del}{\del x_l}(\det G_s)
\right)
dz_i\wedge d\bar{z}_j, \notag \\
\omega&=\frac{\sqrt{-1}}{2}\sum_{i,j}\frac{\del^2 F_s}{\del \xi_i 
\del \xi_j} dz_i\wedge d\bar{z}_j
=\frac{\sqrt{-1}}{2}\sum_{i,j}G_s^{ij} dz_i\wedge d\bar{z}_j
\label{eq_omega.y}
\end{align}

\subsection{Prequantum line bundles on toric symplectic manifolds}\label{sec_preq_line_bdl.h}
The Delzant construction 
of toric symplectic manifold $(X_P,\omega)$ 
also gives 
the prequantum line bundle 
$(\pi\colon L\to X_P,\nabla,h)$. 
See \cite{oda1988} or 
Section 2.2 of \cite{BFMN2011}. 
Since $X_P$ is always simply-connected by 
\cite{fulton1993}, hence $H^1(X_P)=\{ 0\}$, 
then the connection $\nabla$ 
with $F^\nabla=-\sqrt{-1}\omega$ 
is uniquely determined up to the bundle 
isomorphisms of $L$. 
Moreover, the hermitian metric 
$h$ with $\nabla h\equiv 0$ is 
also determined uniquely up to 
a multiplicative constant.

Next we consider the local description of 
prequantum line bundle. 
A {\it face of codimension $m$} of $P$ 
is a subset of $P$ 
written as 
\begin{align*}
\{ x\in P;\, {}^t\!\nu_r\cdot x= \lambda_r
\mbox{ for }r=i_1,\ldots,i_m\}
\end{align*}
for some $1\le i_1<\cdots <i_m\le d$. 
Let $b\in P$ be an interior point of 
a face of codimension $m$ 
and $b'\in P$ be one of 
the vertex of this face. 
Here, the vertex means the face of 
codimension $n$. 
Then there is an affine transformation 
$x\mapsto Ax+a$ of $\R^n$ 
by some 
$A\in GL_n\Z$ and $a\in \Z^n$ 
such that we may suppose 
$b'=0$, $P\subset \R_{\ge 0}^n$, 
$b=(b_1,\ldots,b_{n-m},0,\ldots,0)$ and 
$b_1,\ldots,b_{n-m}$ are positive. 
Note that $P\subset \Z^n$ is contained 
in $\Z^n$ again after the affine transformation. 
Let $F_i:=\{ x_i=0\}\cap P$, 
$P':=\breve{P}\cup F_1\cup\cdots\cup F_n$ 
and $U_b=\mu_P^{-1}(P')$. 
Then $U_b$ is diffeomorphic to $\C^n$. 
Here, the action angle coordinate 
$(x,\theta)$ is defined on 
$U_b\setminus (\bigcup_i \{ x_i=0\})$, 
however, $x_id\theta^i$ can be extended 
to the $1$-form on $U_b$. 
The following Proposition 
\ref{loc_desc_conn.h} and 
Corollary \ref{BS_in_polytope.h} are 
well-known, however, we 
give the proof for the 
reader's convenience. 
\begin{prop}\label{loc_desc_conn.h}
Let $(L,\nabla,h)$ be a prequantum bundle 
on $(X_P,\omega)$. 
There is a bundle isomorphism 
$\Phi\colon U_b\times \C\to L|_U$ 
such that $\Phi^*\nabla = d-\sqrt{-1}x_id\theta^i$. 
Here, $d$ is the connection on $U_b\times \C$ 
which makes the section 
$e\colon p\mapsto (p,1)$ parallel. 
\label{local_conn.h}
\end{prop}
\begin{proof}
Since $U_b=\C^n$, $L|_{U_b}$ is trivial 
as a complex line bundle. 
Then there exists a trivialization 
$L|_{U_b}\cong U_b\times \C$ such that 
the section $e\colon p\mapsto (p,1)$ 
satisfies $h(e,e)\equiv 1$. 
Under the identification, 
we may write $\nabla=d-\sqrt{-1}\gamma$ 
for some $\gamma\in\Omega^1(U_b)$ 
with $d\gamma = \omega =dx_i\wedge d\theta^i$. 
Then $\gamma-x_id\theta^i$ is a closed 
$1$-form on $U_b$. 
Since $H^1(U_b,\R)=\{ 0\}$, 
there is $f\in C^\infty(U_b)$ such that 
$\gamma-x_id\theta^i=df$. 
Then by taking the bundle isomorphism 
$e^{\sqrt{-1}f}$, we have 
$\nabla = d-\sqrt{-1}x_id\theta^i$. 
\end{proof}
\begin{cor}\label{BS_in_polytope.h}
$b\in P$ is a Bohr-Sommerfeld point of 
level $k$ iff $b\in P\cap \frac{1}{k}\Z^n$. 
\end{cor}
\begin{proof}
For the simplicity we show the case of $k=1$. 
For the general case, apply 
the following argument 
to $L^k$ equipped 
with the connection induced by $\nabla$. 
For $b\in P$ take $U_b$ and 
the trivialization $L|_{U_b}\cong U_b\times \C$ 
as in Proposition \ref{local_conn.h}. 
Notice that $x_id\theta^i|_{\mu_P^{-1}(b)}$ 
is a closed $1$-form since 
$\mu_P^{-1}(b)$ is an isotropic submanifold. 
Then the holonomy group 
${\rm Hol}(L|_{\mu_P^{-1}(b)},\nabla|_{\mu_P^{-1}(b)})$ is generated by 
\begin{align*}
&\quad\ \left\{ \exp\left(\sqrt{-1}\int_C x_id\theta^i|_{\mu_P^{-1}(b)}\right)\in S^1;\, 
C\in H_1(\mu_P^{-1}(b),\Z)\right\}\\
&= \left\{ \exp\left( 2\pi\sqrt{-1}x_i\right)\in S^1;\, 
i=1,\ldots n\right\}.
\end{align*}
Therefore, 
${\rm Hol}(L|_{\mu_P^{-1}(b)},\nabla|_{\mu_P^{-1}(b)})$ is trivial iff $x_1,\ldots,x_n$ are integers. 
\end{proof}

\section{Limit spaces}\label{sec_lim_sp.y}
In this section, we describe the pointed $S^1$-equivariant measured Gromov-Hausdorff limits of the frame bundle of $L$ by the family $\{J_s\}_{s > 0}$ as $s \to 0$. 
The main results of this section are Proposition \ref{prop_lim_kBS.y} and Proposition \ref{prop_lim_nonBS.y}, corresponding to the case where the basepoint belongs to a Bohr-Sommerfeld fiber of level $l$ for some $l \in \Z_{>0}$, and otherwise, respectively.  

For simplicity, we first analyze at Bohr-Sommerfeld fiber of level one. 
Fix a Bohr-Sommerfeld point $b \in P \cap \mathbb{Z}^n$. 
Assume that $b$ is an 
interior point of a codimension $m$ face of $P$.  
By a coordinate change, we may assume $b = 0 \in \mathbb{Z}^n$, and near $b$, $P$ is locally defined as $\{x \in \mathbb{R}^n \ | \ x_i \geq 0 \ (i = 1, \cdots m)\}$. 
Let us equip $\breve{X}_P$ with the action-angle coordinate $\breve{X}_P \simeq \breve{P} \times \mathbb{T}^n \in (x, \theta)$ so that
$\omega=\sum_i dx_i\wedge d\theta^i$. 

We can take a neighborhood $\tilde{W}$ of $\mu_P^{-1}(b) \subset X_P$ with coordinate $\tilde{W} = B_\epsilon(0)^m \times (-\epsilon, \epsilon)^{n-m} \times \mathbb{T}^{n-m}$, where $B_\epsilon(0) \subset \mathbb{C}$ is the $\epsilon$-ball around $0$.  
We denote $W := \mu_P(\tilde{W}) \subset P$. 
We can trivialize $L$ on $\tilde{W}$ so that $\nabla = d -\sqrt{-1} {}^t\!xd\theta$. 

Let $G_s := \mathrm{Hess}(v_P + \varphi + s^{-1}\psi)$. 
We have $G_s = \frac{1}{2}X_m^{-1} + s^{-1}A+B$ where $A = \mathrm{Hess}(\psi)$, 
\begin{equation}\label{eq_matX}
X_m^{-1} = 
\begin{pmatrix}
\frac{1}{x_1}& & & & \\
&\ddots & & & \\
& & \frac{1}{x_m}& & \\
& & & 0 & \\
& & & & \ddots 
\end{pmatrix}, 
\end{equation}
and $B = \mathrm{Hess}(v_P + \varphi) - \frac{1}{2}X_m^{-1}$. 
Note that the matrix-valued functions $A, B \in C^\infty(W )\otimes M_n(\mathbb{R})$ are bounded over $W$. 
The metric on the frame bundle $S$ of $L$ induced by $g_s$ is written as
\[
\hat{g}_s = (dt - {}^t\!xd\theta)^2 + {}^t\!dxG_s dx + {}^t\!d\theta G_s^{-1}d\theta. 
\]

First we begin with an easy estimate. 
\begin{prop}\label{prop_ball.h}
Let $b \in P$ and 
$p_b \in \mu_P^{-1}(b)$. 
Fix an action-angle coordinate around $\mu_P^{-1}(b)$ as above. 
Denote by $B(b,R)$ the Euclidean ball of radius $R$ in $\R^n$ centered at $b$. 
\begin{itemize}
\setlength{\parskip}{0cm}
\setlength{\itemsep}{0cm}
 \item[$({\rm i})$] 
There are constants 
$C,\delta_0>0$ such that
\begin{align*}
    B_{g_s}(p_b,r)
    \subset \mu_P^{-1}(B(b,\sqrt{s}Cr))
\end{align*}
for any $r,s>0$ with $\sqrt{s}r< \delta_0$. 
 \item[$({\rm ii})$] For any $r>0$ 
 there is a constant $s_{b,r}>0$ such that 
\begin{align*}
    \mu_P^{-1}(B(b,sr))
    \subset B_{g_s}(p_b,r)
\end{align*}
holds for any $0<s\le s_{b,r}$. 
\end{itemize}
\end{prop}
\begin{proof}
$({\rm i})$ 
Let $p\in X_P$ and 
$c\colon [0,1]\to X_P$ be a 
piecewise smooth path such that 
$c(0)=p_b$ and $c(1)=p$. 
Using the action angle coordinate, 
we write 
$c(\tau)=(x(\tau),\theta(\tau))$. 
Now we have 
\begin{align}
g_s &= {}^t\!d\theta G_s^{-1} 
d\theta 
+ {}^t\!dxG_sdx.
\end{align} 
Let $a>0$ be the minimum 
of the eigenvalues of $A(0)$ 
and take $\delta'>0$ such that 
$A(x)\ge \frac{a}{2} I_n$ holds 
on $B(b,\delta')$. 
If the image of 
$\mu_P\circ c$ 
is contained in 
$B(b,\delta')$, 
the length $L(c)$ of 
$c$ with respect to $g_s$ 
is estimated as 
\begin{align*}
    L(c)
    =\int_0^1|c'(\tau)|_{g_s}d\tau
    &\ge  \int_0^1\left( \sum_{i,j}x_i'x_j'G_{s,ij}
    \right)^{1/2}d\tau\\
    &\ge  s^{-1/2}\int_0^1\left( \sum_{i,j}x_i'x_j'A(x)_{ij}
    \right)^{1/2}d\tau\\
    &\ge  \frac{s^{-1/2}a}{2}\| \mu_P(p)\|. 
\end{align*}

If the image of 
$\mu_P\circ c$ 
is not contained in 
$B(b,\delta')$, 
let $\tau_0\in [0,1]$ 
be the minimum value 
which satisfies 
$\|\mu_P\circ c(\tau_0)\|\ge 
\delta'$. 
Then
\begin{align*}
    L(c)
    =\int_0^1|c'(\tau)|_{g_s}d\tau
    \ge \int_0^{\tau_0}|c'(\tau)|_{g_s}d\tau
    \ge \frac{s^{-1/2}a\delta'}{2}.
\end{align*}
Therefore, we obtain 
\begin{align*}
    d_{g_s}(p,p_b)=
    \inf_c L(c)
    \ge \min\left\{\frac{s^{-1/2}a}{2}\| \mu_P(p)\|,\ 
    \frac{s^{-1/2}a\delta'}{2}
    \right\}.
\end{align*}
Suppose $\sqrt{s}r<\frac{a\delta'}{2}$ and 
$p\in B_{g_s}(p_b,r)$. 
Then we have 
$\frac{\sqrt{s}a}{2}\| \mu_P(p)\|<r$, 
which gives 
the assertion if we put 
$C=\frac{2}{a}$ and 
$\delta=\frac{a\delta'}{2}$. 

$({\rm ii})$ 
We estimate the length of 
the following 
two types of paths connecting 
$p_b$ and $p\in 
\mu_P^{-1}(B(b,\delta))$. 

First of all, let 
$c_1(\tau):=(b+\tau v,\theta)$ 
for some fixed $v\in B(0,\delta)$ 
and $\theta\in \mathbb{T}^n
=\R^n/(2\pi \Z)^n$. 
Take $\delta'>0$ and constants $N_1, N_2 >0$ such that 
$A(x)\le NI_n$ and 
$ B(x)\le NI_n$ for 
any $x\in B(b,\delta')$. 
Then we have 
\begin{align*}
    L(c_1)
    &=\int_0^1\left( 
    {}^t\!vG_s(\tau v) v\right)^{1/2}d\tau\\
    &\le \int_0^1\left( 
    {}^t\!vX_m^{-1}(\tau v) v\right)^{1/2}d\tau
    +\sqrt{2N(s^{-1}+1)}\| v\|\\
    &\le \int_0^1\left( 
    \frac{v_1+\cdots v_m}{2\tau}\right)^{1/2}d\tau
    +\sqrt{2N(s^{-1}+1)}\| v\|\\
    &\le \sqrt{ 2n\| v\|}
    +\sqrt{2N(s^{-1}+1)}\| v\|.
\end{align*}
Next we put $c_2(\tau)=(b,\theta+\tau v)$, 
where $0\le v_i \le \pi$. 
Then 
\begin{align*}
    L(c_2)
    &=\int_0^\pi\left( 
    {}^t\!vG_s^{-1}(b) v\right)^{1/2}d\tau.
\end{align*}
Here, $G_s^{-1}(x)$ can be 
extended to $x_i=0$ 
for $i=1,\ldots,m$ and 
one can check 
$G_s^{-1}(b)=O(s)$. 
Therefore, there is $C_0>0$ 
such that  
\begin{align}
    L(c_2)&\le C_0\sqrt{s}.
    \label{ineq_fiber.h} 
\end{align}
Now, let $p\in 
\mu_P^{-1}(B(b,\delta'))$. 
Then one can construct a path 
connecting $p_b$ and $p$ 
by combining $c_1$ and $c_2$, 
then one can see that 
\begin{align*}
    d_{g_s}(p_b,p)
    \le \sqrt{2n\| \mu_P(p)-b\|} + \sqrt{2N}\sqrt{s^{-1}+1}\| \mu_P(p)-b\| + C_0\sqrt{s}.
\end{align*}

Fix $r > 0$ and take $s > 0$ such that $sr< \delta'$. 
We have
\begin{align}\label{ineq_1.y}
    d_{g_s}(p_b,p)
    &\le \sqrt{2nsr} + \sqrt{2N}\sqrt{s^{-1}+1}sr + C_0\sqrt{s} \notag\\
    &= \sqrt{s}(\sqrt{2nr} + \sqrt{2N}\sqrt{1 + s}r + C_0)
\end{align}
Then it is clear that there exists a constant $s_{b, r} > 0$ such that \eqref{ineq_1.y} is smaller than $r$ for all $0 < s \le s_{b, r}$, so we get the result. 
\end{proof}

It is easy to see the following estimate on the diameters of the fiber of $\mu_P$. 
\begin{lem}\label{lem_diam_fiber.y}
We have
\begin{align*}
    \sup_{0 < s \le 1, b \in P}s^{-1/2} \mathrm{diam}(\mu_P^{-1}(b)) <+ \infty. 
\end{align*}
\end{lem}

\begin{rem}\label{rem_ball_submersion.h}
\normalfont
Since $\pi\colon (S,\hat{g}_s)
\to (X_P, g_s)$ is a Riemannian 
submersion
and the diameters of the fibers 
are at most $2\pi$, 
we have 
\begin{align*}
    \pi^{-1}(B_{g_s}(p,r-2\pi))\subset
    B_{\hat{g}_s}(u,r)\subset
    \pi^{-1}(B_{g_s}(p,r))
\end{align*}
holds for any $p\in X_P$ and 
$u\in S$ with $\pi(u)=p$. 
\end{rem}

Now we proceed to describe the limit space. 
We consider the cone $\mathcal{C}_b(\psi) \subset \R^n$ defined in Definition \ref{def_cone.y}. 
In our coordinate, we have
$A(0) = \mathrm{Hess}(\psi)_b$ and $\mathcal{C}_b(\psi) = C_m(A(0)) = A(0)^{1/2}(\mathbb{R}_{\geq 0}^m \times \mathbb{R}^{n-m})$ (see \eqref{eq_std_cone.y}). 
Let $g_\infty$ be the metric on $\mathcal{C}_b(\psi)\times S^1$ defined by
\begin{equation}\label{eq_met_infty}
g_\infty := \frac{1}{1 +  \|\xi\|^2}(dt)^2+ {}^t\!d\xi\cdot d\xi 
\end{equation}
Here we use the coordinate $(\xi_1, \cdots, \xi_n, t) \in \mathcal{C}_b(\psi)\times S^1 \subset \R^n \times S^1$. Let $S^1$ act on $\mathcal{C}_b(\psi) \times S^1$ by $e^{\sqrt{-1}\tau} \cdot (\xi, e^{\sqrt{-1}t}) = (\xi, e^{\sqrt{-1}(t + \tau)})$, and regard $\mathcal{C}_b(\psi) \times S^1, g_{\infty}, \det (\mathrm{Hess}(\psi)_b))^{-1/2}d\xi dt)$ as a metric measure space with an isometric $S^1$-action. 
\begin{prop}\label{prop_S^1-pmGH.h}
Let $b \in P \cap \mathbb{Z}^n$ be a Bohr-Sommerfeld point. 
Choose any lift $u_b \in S$ of $b$. 
The family of pointed metric measure spaces with the isometric $S^1$-action
\[
\{(S, \hat{g}_s, s^{-n/2}\nu_{\hat{g}_s}, u_b)\}_s
\]
converges to $(\mathcal{C}_b(\psi) \times S^1, g_{\infty}, \det (\mathrm{Hess}(\psi)_b)^{-1/2}d\xi dt, (0, 1))$ as $s \to 0$ in the sense of $S^1$-equivariant pointed measured Gromov-Hausdorff topology. 
Here the $S^1$-action on $\mathcal{C}_b(\psi) \times S^1$ is given by $e^{\sqrt{-1}\tau} \cdot (\xi, e^{\sqrt{-1}t}) = (\xi, e^{\sqrt{-1}(t + \tau)})$. 
\end{prop}
\begin{proof}
We proceed similarly as in \cite[Theorem 7.16]{hattori2019}, but since we are assuming that the metric tensors only depend on the action variables, the proof is simpler here. 
We use the coordinate as above. 
Fix $s >0$. 
On $\tilde{W}$, we have 
\begin{align*}
\hat{g}_s &=  dt^2 - 2{}^t\!xd\theta dt + {}^t\!d\theta(G_s^{-1} +  x{}^t\!x)d\theta + {}^t\!dxG_sdx\\
&={}^t\!(K^{1/2}d\theta -  K^{-1/2}xdt)(K^{1/2}d\theta -  K^{-1/2}xdt)\\
&\quad + (1 - {}^t\!xK^{-1}x)dt^2 +
{}^t\!dxG_sdx,
\end{align*}
where 
$K := G_s^{-1} +  x{}^t\!x$. 

First we show 
the convergence, 
\begin{align}\label{eq_lim_base.y}
(W \times S^1, (1 -  \ {}^t\!xK^{-1}x)dt^2 + {}^t\!dxG_sdx, (b, 1)) \xrightarrow{S^1-\mathrm{pmGH}} (C_m(A(0)) \times S^1, g_\infty, (0, 1)) \ (s \to 0). 
\end{align}
Let us define $z := s^{-1/2}x$. 
We define $Z_m$ similarly as (\ref{eq_matX}). 
For each $z \in \R^m_{\ge 0} \times \R^{n-m}$, for $s>0$ small enough so that $s^{1/2}z \in W$, we have
\begin{align*}
G_s &= \frac{s^{-1/2}}{2}Z_m^{-1} + s^{-1}A(s^{1/2}z) + B(s^{1/2}z) \\
K &= s((A(s^{1/2}z)+ \frac{s^{1/2}}{2}Z_m^{-1}+ sB(s^{1/2}z) )^{-1} +  z{}^t\!z) \\
1 - \ {}^t\!xK^{-1}x&= 1 -  {}^t\!z ((A(s^{1/2}z)  + \frac{s^{1/2}}{2}Z_m^{-1}+ sB(s^{1/2}z) )^{-1} +  z{}^t\!z)^{-1}z \\
& \to 1 -  {}^t\!z(A(0)^{-1} +  z{}^t\!z )^{-1}z \ (s \to 0) \\
 {}^t\!dxG_sdx &= {}^t\!dz(A(s^{1/2}z)  + \frac{s^{1/2}}{2}Z_k^{-1}+ sB(s^{1/2}z) )dz \\
 &\to {}^t\!dzA(0)dz \ (s \to 0). 
\end{align*}
Let us define $\xi := A(0)^{1/2}z$. 
We have
\begin{align*}
1 -  {}^t\!z(A(0)^{-1} +  z{}^t\!z )^{-1}z&= \frac{1}{1 +  \|\xi\|^2} \\
{}^t\!dzA(0)dz &= {}^t\!d\xi d\xi. 
\end{align*}
For each $s > 0$, we define the map
\[
F'_s : W \times S^1 \to C_m(A(0)) \times S^1: (x, t) \mapsto (\xi = s^{-1/2}A(0)^{1/2}x, t). 
\]
Also note that there exists a constant $R > 0$ such that $\mathrm{Im}(F'_s) \supset B_{g_\infty}((0, 1), s^{-1/2}R)$ for all $0<s\le 1$. 
From the computations above, we get the convergence (\ref{eq_lim_base.y})
given by the approximation maps $\{F'_s\}_{s \ge 0}$. 

Away from the faces of $W$, $\hat{g}_s$ is a submersion metric with respect to the submersion $S|_{\mu_P^{-1}(\mathring{W}) }\to \mathring{W} \times S^1$, and the diameters of the fibers of the map $S|_{\mu_P^{-1}(W) }\to {W} \times S^1$ are uniformly bounded by $ O(K^{1/2}) = O(s^{1/2})$ by Lemma \ref{lem_diam_fiber.y}. 
Combining these, we get the asymptotically $S^1$-equivariant pointed Gromov-Hausdorff convergence, 
\[
(S, \hat{g}_s, u_b) \xrightarrow{S^1-\mathrm{pGH}} (C_m(A(0)) \times S^1, g_\infty, (0, 1)) \ (s \to 0), 
\]
given by the approximation maps 
\begin{align}\label{eq_approx_map.y}
F_s \colon S|_{\mu_P^{-1}(W)} \to C_m(A(0)) \times S^1: (x, \theta, t) \mapsto (\xi = s^{-1/2}A(0)^{1/2}x, t). 
\end{align}

Now we look at measures. 
We have $\nu_{\hat{g}_s} = dxd\theta dt$. 
For any $f \in C^\infty_c(C_m(A(0)) \times S^1)$, for $s$ small enough we have
\begin{align*}
\int_{S} F_s^*f \nu_{\hat{g}_s} &=  \int_{W \times S^1} f(s^{-1/2}A(0)^{1/2}x, t)dxdt \\
&= s^{n/2}\det (A(0))^{-1/2} \int_{C_m(A(0)) \times S^1}f(\xi, t)d\xi dt. 
\end{align*}
So we get the result. 
\end{proof}

For Bohr-Sommerfeld fibers of level $l$ for general $l$, as in the argument in \cite[Section 6]{hattori2019}, we take an $l$-fold covering of a neighborhood of $u_b$ and reduce to the case of Bohr-Sommerfeld fibers of level one, as follows. 
Let $b\in P \cap \frac{\mathbb{Z}^n}{l}$ be a strict $l$-Bohr-Sommerfeld point. 
Assume that $b$ is an interior point of a codimension $m$ face of $P$.  
By a coordinate change of the form $x \mapsto Ax + c$ with $A \in GL_n\Z$ and $c \in \frac{\Z^n}{l}$, we may assume $b = 0 \in \frac{\mathbb{Z}^n}{l}$, and near $b$, $P$ is locally defined as $\{x \in \mathbb{R}^n \ | \ x_i \geq 0 \ (i = 1, \cdots m)\}$. 
We can take a neighborhood $\tilde{W}$ of $\mu_P^{-1}(b) \subset X_P$ with coordinate $\tilde{W} = B_\epsilon(0)^m \times (-\epsilon, \epsilon)^{n-m} \times \mathbb{T}^{n-m}$, where $B_\epsilon(0) \subset \mathbb{C}$ is the $\epsilon$-ball around $0$.  

Now we fix some notations. 
Let $\Phi\colon \Z^{n-m} \to \Z / l\Z$ be a homomorphism of $\Z$-modules. 
Then we have the natural projection 
\begin{align*}
\R^{n-m}/{2\pi\,\rm Ker }\Phi\to  \T^{n-m} 
\end{align*}
which gives a covering space and 
a covering map 
\begin{align*}
\tilde{W}_\Phi:=B_\epsilon(0)^m \times (-\epsilon, \epsilon)^{n-m}\times \left( 
\R^{n-m}/2\pi\,{\rm Ker}\,\Phi\right),\quad 
p_\Phi\colon \tilde{W}_\Phi\to \tilde{W}. 
\end{align*}
From now on we denote by 
$\theta$ the element of $\R^{n-m}/2\pi\,{\rm Ker}\,\Phi $ 
or $\T^{n-m}$ for the simplicity, 
if there is no fear of confusion. 
If we take $\mathbf{w}\in\Z^{n-m}$ 
then 
\[ 
\left.
\begin{array}{cccc}
\beta(\Phi(\mathbf{w})):
& \tilde{W}_\Phi & \rightarrow & \tilde{W}_\Phi \\
&(x,\theta) & \mapsto & (x,\theta+2\pi\mathbf{w})
\end{array}
\right.
\]
gives the action of ${\rm Im}\,\Phi$ 
on $\tilde{W}_\Phi$, which is 
the deck transformations of $p_\Phi$. 

Analogously to \cite[Proposition 6.1]{hattori2019}, we have the following. 
\begin{prop}\label{prop_kBS_trivialization.y}
Let $b\in P \cap \frac{\mathbb{Z}^n}{l}$ be a strict $l$-Bohr-Sommerfeld point.  
Then, if we take a coordinate change as above, there are surjective 
homomorphism $\Phi\colon \Z^{n-m} \to \Z / l\Z$ 
and
$E\in C^\infty(\tilde{W}_\Phi; p_\Phi^*L)$ such that 
$h(E,E)\equiv 1$ and 
$\nabla E = -\sqrt{-1}x_id\theta^i\otimes E$. 
Moreover, 
the deck transformations of $p_\Phi$ 
satisfy 
$\beta(j)^*E = e^{\frac{2j\sqrt{-1}\pi}{l}}E$ 
for $j\in\Z/ l\Z$. 
\end{prop}
Thus we can apply the argument above for the case for Bohr-Sommerfeld fibers of level one to the line bundle $p_\Phi^*L \to \tilde{W}_\Phi$. 
Using Proposition \ref{prop_kBS_trivialization.y}, we have a trivialization $p^*_\Phi S = S(p^*_\Phi L) \simeq \tilde{W}_\Phi \times S^1$. 
The deck transformations of 
\begin{align*}
p_\Phi\colon p^*_\Phi S \to S|_{\tilde{W}}
\end{align*}
are identified with 
\begin{align}
j\cdot(x,\theta,e^{\sqrt{-1}t})
:=(x,\theta+2\pi j\mathbf{w}_0, e^{\sqrt{-1}(t-\frac{2j\pi}{l})})
\quad(j\in\Z/l\Z), 
\label{deck}
\end{align}
where $\mathbf{w}_0\in\Z^{n-m}$ is 
taken such that $\Phi(\mathbf{w}_0)
=1\in\Z/l\Z$. 
Let $g_\infty$ be the metric on $\mathcal{C}_b(\psi) \times S^1$ defined in (\ref{eq_met_infty}). 
Choosing any lift $\tilde{u}_b \in \tilde{W}_\Phi \times S^1$ of $b$, we have an $S^1$-equivariant pointed measured Gromov-Hausdorff convergence, 
\begin{align*}
(\tilde{W}_\Phi \times S^1, p_\Phi^* \hat{g}_s, s^{-n/2}\nu_{p_\Phi^*\hat{g}_s}, \tilde{u}_b) \xrightarrow{s \to 0} 
(\mathcal{C}_b(\psi) \times S^1, g_{\infty}, \det (\mathrm{Hess}(\psi)_b)^{-1/2}d\xi dt, (0, 1)). 
\end{align*}

Let $\Z/l\Z$ act on $\mathcal{C}_b(\psi) \times S^1$ by $j \cdot (\xi, e^{\sqrt{-1}t})= (\xi, e^{\sqrt{-1}(t-\frac{2j\pi}{l})})$, and denote the quotient map by $p_l : \mathcal{C}_b(\psi) \times S^1 \to \mathcal{C}_b(\psi) \times S^1$.  
Let $g_{l,\infty}$ be the metric on $\mathcal{C}_b(\psi)\times S^1$ defined by
\[
g_{l,\infty} := \frac{1}{l^2(1 +  \|\xi\|^2)}(dt)^2+ {}^t\!d\xi\cdot d\xi 
\]
We have a commutative diagram, 
\[ 
\left.
\begin{array}{ccc}
(\tilde{W}_\Phi
\times S^1, p_\Phi^* \hat{g}_s, ls^{-n/2}\nu_{p_\Phi^*\hat{g}_s}, \tilde{u}_b) & \xrightarrow{S^1\mathchar`-{\rm pmGH}} & (\mathcal{C}_b(\psi) \times S^1, g_{\infty}, l\nu_\infty, (0, 1)) \\
p_\Phi \downarrow & & p_l \downarrow \\
(S, \hat{g}_s, s^{-n/2}\nu_{\hat{g}_s}, u_b) & \xrightarrow{S^1\mathchar`-{\rm pmGH}} & 
(\mathcal{C}_b(\psi) \times S^1, g_{l, \infty}, \nu_\infty, (0, 1))
\end{array}
\right. 
\]
where $\nu_\infty := \det (\mathrm{Hess}(\psi)_b)^{-1/2}d\xi dt$
Thus we get the following. 

\begin{prop}\label{prop_lim_kBS.y}
Let $b \in P \cap \frac{\mathbb{Z}^n}{l}$ be a strict $l$-Bohr-Sommerfeld point. 
Choose any lift $u_b \in S$. 
The family of pointed metric measure spaces with the isometric $S^1$-action
\[
\{(S, \hat{g}_s, s^{-n/2}\nu_{\hat{g}_s}, u_b)\}_s
\]
converges to $(\mathcal{C}_b(\psi) \times S^1, g_{l, \infty}, \det (\mathrm{Hess}(\psi)_b)^{-1/2}d\xi dt, (0, 1))$ as $s \to 0$ in the sense of $S^1$-equivariant pointed measured Gromov-Hausdorff topology. 
Here the $S^1$-action on $\mathcal{C}_b(\psi)\times S^1$ is given by 
$e^{\sqrt{-1}\tau} \cdot (\xi, e^{\sqrt{-1}t}) = (\xi, e^{\sqrt{-1}(t + l\tau)})$. 
\end{prop}

For fibers which are not Bohr-Sommerfeld of level $l$ for any $l$, we have the following. 

\begin{prop}\label{prop_lim_nonBS.y}
Let $b \in P$ be a point which is not Bohr-Sommerfeld of level $l$ for any $l$. 
Choose any lift $u_b\in S$. 
The family of pointed metric measure spaces with the isometric $S^1$-action
\[
\{(S, \hat{g}_s, s^{-n/2}\nu_{\hat{g}_s}, u_b)\}_s
\]
converges to $(\mathcal{C}_b(\psi), {}^t\!d\xi\cdot d\xi, \det (\mathrm{Hess}(\psi)_b)^{-1/2}
d\xi , 0)$
as $s \to 0$ in the sense of $S^1$-equivariant pointed measured Gromov-Hausdorff topology. 
Here the $S^1$ acts on $\mathcal{C}_b(\psi)$ trivially. 
\end{prop}
\begin{proof}
The proof is analogous to the one in \cite[Section 9]{hattori2019}. 
Note that in \cite[Proposition 9.1 and Proposition 9.2]{hattori2019}, we have not assumed that the fibration is regular. 
\end{proof}

\section{Analysis of the limit space}\label{sec_analysis_limit.y}
Let $A \in M_n(\R)$ be a positive definite matrix. 
In this section we analyze the Laplacian of the metric measure space $(C_m(A) \times S^1, g_{l,\infty}, d\xi dt)$. 
Remark that, if we multiply the measure by a positive constant $a > 0$ and consider $(C_m(A) \times S^1, g_{l,\infty}, ad\xi dt)$, the resulting Laplacians are equivalent under the obvious identification of $L^2$-spaces, so it is enough to set $a = 1$. 

Recall that we have defined
\begin{align*}
C_m(A) &= A^{1/2}(\mathbb{R}^m_{\geq 0} \times \mathbb{R}^{n-m}) \subset \mathbb{R}^n \\
g_{l,\infty} &= \frac{1}{l^2(1 + \|\xi\|^2)}(dt)^2+ {}^t\!d\xi\cdot d\xi . 
\end{align*}
Set $X_{l, m, A}:=(C_m(A) \times S^1, g_{l,\infty},  d\xi dt)$. 
Let us denote the Laplacian on this metric measure space by $\Delta_{l, m, A}$. 
By Definition \ref{def_mm_lap.y}, this operator is defined so that 
\begin{equation}\label{eq_dom}
\mathcal{D}(\Delta_{l, m,A}) = \left\{
\begin{array}{c|l}
f \in H^{1,2}(X_{l,m,A}) & \exists h \in L^2(X_{l,m,A}), \forall \phi \in H^{1,2}(X_{l,m,A}), \\
& \int_{C_m(A) \times S^1} \langle df, d\phi \rangle_{g_{l, \infty}} d\xi dt = \int_{C_m(A) \times S^1}  h\phi d\xi dt 
\end{array}
\right\}, 
\end{equation}
and for $f \in \mathcal{D}(\Delta_{l,m, A})$, we have $\Delta_{l,m, A} f = h$ for $h$ appearing in the above equation. 

\begin{prop}
A function $f \in C^\infty_c(C_m(A) \times S^1)$ is in $\mathcal{D}(\Delta_{l,m, A})$ if and only if $\frac{\partial}{\partial \bf{n}}f = 0$ on all the faces of $C_m(A)$. 
Here we denoted by $\bf{n}$ the normal vector for a codimension $1$ face of $C_m$ with respect to the Euclidean metric on $\mathbb{R}^n$. 
For $f \in \mathcal{D}(\Delta_{l,m, A})\cap C_c^\infty(C_m(A) \times S^1)$, we have
\begin{align}\label{eq_harm_ocillator_cone.y}
\Delta_{l,m, A} f = \Delta_{\mathbb{R}^n}f - {(1 +  \|\xi\|^2)}\frac{\partial^2}{\partial t^2}f. 
 \end{align}
In other words, the operator $\Delta_{l, m, A}$ is the closure of the differential operator appearing in the right hand side of \eqref{eq_harm_ocillator_cone.y} with the Neumann boundary condition. 
\end{prop}
\begin{proof}
In general, let $M$ be a manifold with boundaries and corners. 
Let $g$ and $\mu$ be a metric and a smooth density on $M$, respectively. 
Let us define the generalization of the Hodge star in this context, $\star_{g, \mu} \in \mathrm{End}(\wedge_\C T^*M)$, by requiring, for all $\alpha, \beta \in C^\infty(M, \wedge^p T^*M)$, 
\[
\langle \alpha, \beta \rangle_{g} d\mu= \alpha \wedge \star_{g, \mu}\beta. 
\]
We have, for $f, h \in C^\infty_c(M)$, 
\begin{align*}
\int_M \langle dh, df \rangle_g d\mu &= \int_M dh \wedge \star_{g, \mu} df \\
&= \int_{\partial M} h \wedge \star_{g, \mu}df - \int_M h \wedge d\star_{g, \mu}df. 
\end{align*}
Apply this to our case, $M = C_m(A) \times S^1$, $g = g_{l,\infty}$ and $\mu =  d\xi dt$. 
By (\ref{eq_dom}), we see that $f \in C_c^\infty(C_m(A) \times S^1)$ is in $\mathcal{D}(\Delta_{l,m, A})$ if and only if $\star_{g, \mu}df|_{\partial M} \equiv 0$. 
It is equivalent to the condition $\frac{\partial}{\partial \bf{n}}f \equiv 0$. 

For such $f$ we have $\Delta_{l,m, A} f = d\star_{g, \mu} d$. 
The calculation is the same as in \cite[Section 5]{hattori2019}. 
\end{proof}

The relation between the above operator and the Laplacian on the metric measure space $(C_m(A), {}^t\! d\xi \cdot d\xi, e^{-k\|\xi\|^2}d\xi)$, is explained as follows (see \cite[Section 8]{hattori2019} for the corresponding explanation in the boundaryless case). 
Let us fix $l$. 
By Proposition \ref{prop_lim_kBS.y}, when we take a limit at a strict $l$-Bohr-Sommerfeld point, we get the limit space of the form $(C_m(A) \times S^1, g_{l,\infty}, d\xi dt, (0, 1))$ with $S^1$-action given by $e^{\sqrt{-1}\tau} \cdot (\xi, e^{\sqrt{-1}t}) = (\xi, e^{\sqrt{-1}(t + l\tau)})$. 
For a positive integer $k\in l\Z$, if we write $k = jl$ we have
\begin{align*}
\left( L^2(C_m(A) \times S^1) \times \mathbb{C} \right)^{\rho_{k}} = \{\phi(\xi)e^{-\sqrt{-1}jt} \ | \ \phi \in L^2(C_m(A), d\xi)\}. 
\end{align*}
This induces the isomorphism 
\begin{align}\label{eq_harm1.y}
L^2(C_m(A),e^{-k\| \xi\|^2}d\xi)
\otimes\C
&\cong 
\left( L^2 (C_m(A)\times S^1,d\xi dt)
\otimes \C\right)^{\rho_{k}} \notag \\
\varphi &\mapsto 
\frac{1}{\sqrt{2\pi }}\varphi \cdot e^{-\frac{k\|\xi\|^2}{2}  -\sqrt{-1}jt}
\end{align}
and the identification of differential operators 
\begin{align*}
\Delta_{C_m(A)}^k
=\sum_{i=1}^n
\left( -\frac{\del^2 }{\del y_i^2} 
+2k y_i\frac{\del }{\del y_i}
\right)
\cong 
\Delta_{l,m, A}^{\rho_k} - (k^2+ kn). 
\end{align*}
The boundary condition is transformed to the condition 
\[
\frac{\partial}{\partial \bf{n}}\varphi = 0
\]
on each face of $C_m(A)$, i.e., the Neumann boundary condition. 
This operator with Neumann boundary condition, still denoted by $\Delta_{C_m(A)}^k$, is the Laplacian on the metric measure space $(C_m(A), {}^t\! d\xi \cdot d\xi, e^{-k\|\xi\|^2}d\xi)$. 
In this way, we can identify the spectral structures, 
\begin{align}\label{eq_harm2.y}
    \left(L^2(C_m(A), e^{-k\|\xi\|^2} d\xi)\otimes \C, \Delta_{C_m(A)}^k \right) \cong \left( \left(L^2(C_m(A) \times S^1,d\xi dt)
\otimes \C\right)^{\rho_{k}}, \Delta_{l,m , A}^{\rho_k} - (k^2+ kn)\right). 
\end{align}

\begin{ex}
In the case where $A = I_n$, as is well-known, we can describe the spectrum of $\Delta_{C_m(I_n)}^j$ explicitly as follows. 
Recall that we have $C_m(I_n) = \R_{\ge 0}^m \times \R^{n-m}$. 
In the case where $(m, n) = (0, 1)$, we know that
\begin{align*}
    \mathrm{Spec}(\Delta_{\R}^k) &= 2k\Z_{\ge 0}, \\
    W(2kN) &= \C \left\{ e^{k\|y\|^2} 
    \left( \frac{\del}{\del y}\right)^N e^{-k\|y\|^2} \right\}, 
\end{align*}
where we denoted by $W(2kN)$ the eigenspace corresponding to the eigenvalue $2kN$ for $N \in \Z_{\ge 0}$. 

In the case where $(m, n) = (1, 1)$, the set of the eigenfunctions of $\Delta_{\R_{\ge 0}}^k$ consists of those of $\Delta_\R^k$ which are even functions, so we have
\begin{align*}
    \mathrm{Spec}(\Delta_{\R_{\ge 0}}^k) &= 4k\Z_{\ge 0}, \\
    W(4kN) &= \C \left\{ e^{k\|y\|^2}
    \left( \frac{\del}{\del y}\right)^{2N} e^{-k\|y\|^2} \right\}. 
\end{align*}
For general $(m, n)$, the operator is 
the product of $m$-copies of $\Delta_{\R_{\ge 0}}^k$ and $(n-m)$-copies of $\Delta_{\R}^k$, so we see that $\mathrm{Spec}(\Delta_{C_m(I_n)}^k) \subset 2k\Z_{\ge 0}$ (equality holds unless $n = m$) and the multiplicity of the eigenvalue $2kN$ is equal to the number of elements $(k_1, \cdots, k_n) \in (\Z_{\ge 0})^n$ which satisfy $2(k_1 + \cdots + k_m) + k_{m+1} + \cdots + k_n = N$. 
\end{ex}

\begin{prop}\label{prop_spec_Gaussian.y}
The Laplacian $\Delta_{C_m(A)}^k$ on the metric measure space $(C_m(A), {}^t\! d\xi \cdot d\xi, e^{-k\|\xi\|^2}d\xi)$ has compact resolvent, and the $0$-eigenspace is one-dimensional spanned by constant functions. 
As a result, if we have $k = jl$, the operator $\Delta_{l,m, A}^{\rho_k}$ on $\left(L^2(C_m(A) \times S^1),d\xi dt)
\otimes \C\right)^{\rho_{k}}$ has compact resolvent, the lowest eigenvalue is $k^2 + kn$, and the corresponding eigenspace is one-dimensional spanned by the function $e^{-\frac{k\|\xi\|^2}{2}  -\sqrt{-1}jt} \in \left(L^2(C_m(A) \times S^1),d\xi dt)
\otimes \C\right)^{\rho_{k}}$. 
\end{prop}
\begin{proof}
First of all, we know that the Gaussian space, $(\R^n, {}^t\! d\xi \cdot d\xi, e^{-k\|\xi\|^2}d\xi)$ is an $RCD(1, \infty)$ space. 
Since the subspace $C_m(A) \subset \R^n$ is geodesic and $\partial C_m(A)$ is of measure zero, we can apply \cite[Theorem 7.2]{AMS2016}, so $(C_m(A), {}^t\! d\xi \cdot d\xi, e^{-k\|\xi\|^2}d\xi)$ is also an $RCD(1, \infty)$ space. 
Since its measure is finite, we see that the Laplacian $\Delta_{C_m(A)}^k$ has compact resolvent. 

If an element $\varphi \in H^{1,2}(C_m(A), {}^t\! d\xi \cdot d\xi, e^{-k\|\xi\|^2}d\xi)$ satisfies $\Delta_{C_m(A)}^k \varphi = 0$, we need to have $d\varphi=0$, so $\varphi$ is a constant function. 

The statement about $\Delta_{l,m,A}^{\rho_k}$ follows from above and identifications \eqref{eq_harm1.y} and \eqref{eq_harm2.y}. 
\end{proof}

\section{Strong spectral convergence}\label{sec_str_spec_conv.y}

In this section, we prove the strong spectral convergence result (which is weaker than compact convergence; see subsection \ref{subsec_spec_conv_general.y}) for the family of spectral structure in Theorem \ref{thm_cpt_conv.h} (equivalently Theorem \ref{thm_main.y}). 
The main result is Proposition \ref{prop_str_conv.y}. 

In subsections \ref{sec_comp_ric.h} and \ref{subsec_est_ric.y}, we compute and estimate the Ricci curvatures of our family of spaces.
This is the most technical part of this paper. 
If we had a uniform lower bound for the Ricci curvatures on the family $\{(S, \hat{g}_s)\}_s$, the strong spectral convergence would follow simply from Fact \ref{fact_str_conv_multiple.y}. 
However, as shown in subsection \ref{subsec_est_ric.y}, we do not have the uniform lower bound for our family in general, and this makes the things complicated. 
Our strategy is to consider the model space $X = \C^{m} \times \R^{n-m} \times \T^{n-m}$ with the standard toric structure equipped with the family of metrics corresponding to $G_s := s^{-1}(Y_m+A)$ for a constant positive definite matrix $A \in M_n(\R)$ (see the first part of subsection \ref{sec_comp_ric.h}). 
We show that, outside the union of the inverse image of codimension-two faces of the moment polytope for this model space, we have the uniform lower bound for the Ricci curvatures (Proposition \ref{prop_lower_bdd_ric.y}). 
This suffices to give the strong spectral convergence for the model space, because the Sobolev capacity of codimension two faces is zero (Lemma \ref{lem_sob_cap.y}). 
In subsection \ref{subsec_str_conv.y}, we prove the strong spectral convergence. 
The proof of Proposition \ref{prop_str_conv.y} is given by reducing the argument to that of the model space. 

\subsection{Computation of Ricci curvature}\label{sec_comp_ric.h}
We compute the Ricci curvature around boundary points of the polytope. 
Take a coordinate as in Section \ref{sec_lim_sp.y}. 
Set $y_j := s/(2x_j)$. 
Consider the matrix
\begin{equation}\label{eq_matY.y}
Y_m = 
\begin{pmatrix}
y_1& & & & \\
&\ddots & & & \\
& & y_m& & \\
& & & 0 & \\
& & & & \ddots 
\end{pmatrix}. 
\end{equation}
Then we have $G_s = s^{-1}(Y + A + sB)$.
If we set
\begin{align*}
    R_{s, jl} := - \sum_h \frac{\del}{\del x_j}G_{s}^{lh} \frac{\del}{\del x_h} \log (\det G_s^{-1}), 
\end{align*}
Then we have $\rho_s = G^{-1}_s R_s/4$. 
The condition $\mathrm{Ric}(g_s) \ge \kappa g_s$ is equivalent to $G^{-1}_sR_s \ge \kappa Q_s$, which is equivalent to 
\begin{align*}
    R_s G_s \ge \kappa G_s. 
\end{align*}
Set $T_s := R_sG_s $. 
Then we have
\begin{align*}
    T_{s, ji} = -\sum_{h,l}G_{s, li}\frac{\del}{\del x_j}\left( G_{s}^{lh}
    \frac{\del}{\del x_h} \log (\det G^{-1}_s)\right). 
\end{align*}

From now on, we consider simplified settings, where
\begin{enumerate}
    \item $X = \C^{m} \times \R^{n-m} \times \T^{n-m}$ with the standard toric symplectic structure $\mu \colon X \to \R_{\ge 0}^{m} \times \R^{n-m}$ and the corresponding coordinate is denoted by $(x_1, \cdots, x_n, \theta_1, \cdots, \theta_n)$, $x_i \ge 0$ for $1 \le i \le m$. 
    \item Let $A \in M_n(\R)$ be a positive definite matrix and set $G_s := s^{-1}(Y_m+A)$, where $y_i = s/(2x_j)$ and $Y_m$ is defined in \eqref{eq_matY.y}. 
    \item Let $g_s$ be a metric on $X$ given by the formula \eqref{eq_omega.y}. 
\end{enumerate}
We compute $T_{s, ji}$ in this case. 
In the below, for simplicity we drop the reference to the parameter $s$ and write $G$ for $G_s$, etc. 
Let us use the following notations. 
\begin{align*}
    \Delta &:= \det(Y + A), \\
    \Delta_{pq}&:=(-1)^{p+q} \det (Y+A)_{pq} \\
    \Delta_{(p_1 p_2; q_1, q_2)} &:= (-1)^{p_1 + p_2 + q_1 + q_2} \det (Y+A)_{p_1 p_2; q_1 q_2}.
\end{align*}
Here, $(Y+A)_{pq}$ denotes the matrix obtained by deleting the $p$-th row and the $q$-th column from $(Y+A)$, and $(Y+A)_{p_1 p_2; q_1 q_2}$ denotes the matrix obtained by deleting the $p_1, p_2$-th rows and the $q_1, q_2$-th columns from $(Y+A)$. 
Note that we have $G^{pq} = s\Delta_{pq}/\Delta$. 

\begin{lem}
We have, for each $1 \le h \le m$,
\begin{align*}
    \frac{\del}{\del x_h}\log (\det G^{-1}) = \frac{y_h^2}{2s}\frac{\Delta_{hh}}{\Delta}. 
\end{align*}
\end{lem}
\begin{proof}
We have 
\[
\frac{\del}{\del x_h} = -\frac{y_h^2}{2s}\frac{\del}{\del y_h}. 
\]
Since we have
\begin{align*}
    \log(\det G^{-1}) &= -\log (\det G) = \log (\det (Y+A)) + \log s, \\
    \frac{\del}{\del y_h}\det (Y+A) &= \Delta_{hh}, 
\end{align*}
We have
\begin{align*}
    \frac{\del }{\del y_h} \log (\det G^{-1}) = -\frac{\del}{\del y_h} \log(\det (Y+A)) = -\frac{\Delta_{hh}}{\Delta}. 
\end{align*}
\end{proof}
Thus we get
\begin{align}\label{eq_T.y}
    T_{ji} &= -\sum_{1 \le h \le m,1 \le l\le n}G_{li}(-\frac{1}{2s})y_j^2\frac{\del}{\del y_j}\left( \frac{s\Delta_{lh}}{\Delta}\cdot \frac{y_h^2}{2s}\frac{\Delta_{hh}}{\Delta}\right) \notag \\
    &= \frac{1}{4s}\sum_{1 \le h \le m,1 \le l\le n}G_{li}y_j^2\frac{\del}{\del y_j}
    \left( \frac{y_h^2\Delta_{lh}\Delta_{hh}}{\Delta^2}\right). 
\end{align}
By a straightforward computation, we have the followings. 
\begin{lem}\label{lem_sum_in_T.y}
For $1 \le j \le m$, we have
\begin{align*}
    &\frac{\del}{\del y_j}
    \left( \frac{y_h^2\Delta_{lh}\Delta_{hh}}{\Delta^2}\right) \\
    &=\begin{cases}
    y_h^2 \left( \frac{\Delta_{(lh;hj)}\Delta_{hh} + \Delta_{lh}\Delta_{(hj;hj)}}{\Delta^2} -\frac{2\Delta_{lh}\Delta_{hh}\Delta_{jj}}{\Delta^3}\right) & (j \neq h, j \neq l) \\
    y_h^2 \left( \frac{ \Delta_{lh}\Delta_{(hj;hj)}}{\Delta^2} -\frac{2\Delta_{lh}\Delta_{hh}\Delta_{jj}}{\Delta^3}\right) & (j \neq h, j = l) \\
    2y_j^2 \left( \frac{\Delta_{lj}\Delta_{jj}}{\Delta^2} - y_j\frac{\Delta_{lj}\Delta_{jj}^2}{\Delta^3} \right) & (j = h). 
    \end{cases}
\end{align*}
\end{lem}
Now we can compute $T_{ji}$. 
\begin{prop}\label{prop_T.y}
For $1 \le i, j \le m$, we have
\begin{align*}
    T_{ji} = \begin{cases}
    -\frac{1}{4s^2}y_j^2 y_i^2 \frac{\Delta_{ij}^2}{\Delta^2} & (j \neq i) \\
    -\frac{1}{4s^2} \sum_{1 \le h \le m, h\neq j}\left( y_j^2y_h^2 \frac{\Delta_{jh}\Delta_{hh}}{\Delta^2} \right)
    +\frac{y_j^3\Delta_{jj}}{2s^2\Delta} - \frac{y_j^4 \Delta_{jj}^2}{2s^2\Delta^2} & (j = i). 
    \end{cases}
\end{align*}
\end{prop}
\begin{proof}
We first show the case when $j \neq i$. 
We fix $j$ and $i$. 
In the right hand side of the equation \eqref{eq_T.y}, we fix $h$ and first take sum over $1 \le l\le n$. 

In the case where $h \neq j$, we have 
\begin{align*}
    \sum_{l}G_{li}y_j^2\frac{\del}{\del y_j}
    \left( \frac{y_h^2\Delta_{lh}\Delta_{hh}}{\Delta^2}\right)
    =y_j^2 y_h^2\sum_{l}G_{li}\frac{\del}{\del y_j}
    \left( \frac{\Delta_{lh}\Delta_{hh}}{\Delta^2}\right). 
\end{align*}
Using Lemma \ref{lem_sum_in_T.y}, we have
\begin{align}\label{eq_h_neq_j.y}
    &\sum_{l}G_{li}\frac{\del}{\del y_j}
    \left( \frac{\Delta_{lh}\Delta_{hh}}{\Delta^2}\right) \notag \\
    &= \left(\frac{\Delta_{(hj; hj)}}{\Delta^2} - \frac{2\Delta_{hh}\Delta_{jj}}{\Delta^3} \right)
    \sum_{l}G_{li}\Delta_{lh}
    +\frac{\Delta_{hh}}{\Delta^2}\sum_{l \neq j} G_{li}\Delta_{(lj; hj)}. 
\end{align}
By definition of $\Delta_{lh}$ and $\Delta_{(lj; hj)}$, we have
\begin{align}\label{eq_cancel.y}
    \sum_{l}G_{li}\Delta_{lh} &=
    \begin{cases}
    \frac{1}{s}\Delta & (i = h) \\
    0 & (i\neq h),
    \end{cases} \\
    \sum_{l \neq j} G_{li}\Delta_{(lj; hj)} &=
    \begin{cases}
    \frac{1}{s}\Delta_{jj} & (i = h) \\
    0 & (i \neq h). 
    \end{cases}\notag
\end{align}
Thus, 
\begin{align*}
    \eqref{eq_h_neq_j.y} = \begin{cases}
    0  & (i \neq h) \\
    \frac{1}{s}\left( \frac{\Delta_{(ij;ij)}}{\Delta^2} - \frac{2\Delta_{ii}\Delta_{jj}}{\Delta^3}\right)\Delta
    +\frac{1}{s}\frac{\Delta_{ii}\Delta_{jj}}{\Delta^2}
    =\frac{1}{s}\left(  \frac{\Delta_{(ij;ij)}}{\Delta} - \frac{\Delta_{ii}\Delta_{jj}}{\Delta^2}\right)
    & (i = h).
    \end{cases}
\end{align*}

In the case where $h=j$, using Lemma \ref{lem_sum_in_T.y} we have
\begin{align*}
    \sum_{l}G_{li}y_j^2\frac{\del}{\del y_j}
    \left( \frac{y_h^2\Delta_{lh}\Delta_{hh}}{\Delta^2}\right)
    =2y_j^3\left(\frac{\Delta_{jj}}{\Delta^2} - y_j \frac{\Delta_{jj}^2}{\Delta^3} \right) \sum_l G_{li}\Delta_{lj} = 0. 
\end{align*}
Combining the above, we get, for $i \neq j$,
\begin{align}
    T_{ji} = \frac{1}{4s^2}y_j^2y_i^2 \left(  \frac{\Delta_{(ij;ij)}}{\Delta} - \frac{\Delta_{ii}\Delta_{jj}}{\Delta^2}\right)
    = -\frac{1}{4s^2}y_j^2y_i^2\frac{\Delta_{ij}^2}{\Delta^2}. \label{eq_off_diag.h}
\end{align}
(For the right equality of \eqref{eq_off_diag.h}, see Remark \ref{rem_minor.y} below. )

Next we show in the case where $i = j$. 
In the right hand side of the equation \eqref{eq_T.y}, we fix $h$ and first take sum over $1 \le l\le n$. 

In the case where $h\neq j$, we have
\begin{align*}
    \sum_{l}G_{lj}y_j^2\frac{\del}{\del y_j}
    \left( \frac{y_h^2\Delta_{lh}\Delta_{hh}}{\Delta^2}\right)
    =y_j^2 y_h^2\sum_{l}G_{lj}\frac{\del}{\del y_j}
    \left( \frac{\Delta_{lh}\Delta_{hh}}{\Delta^2}\right). 
\end{align*}
Using Lemma \ref{lem_sum_in_T.y}, we have
\begin{align}\label{eq_h_neq_j_hoge.y}
    &\sum_{l}G_{lj}\frac{\del}{\del y_j}
    \left( \frac{\Delta_{lh}\Delta_{hh}}{\Delta^2}\right) \notag \\
    &= \left(\frac{\Delta_{(hj; hj)}}{\Delta^2} - \frac{2\Delta_{hh}\Delta_{jj}}{\Delta^3} \right)
    \sum_{l}G_{lj}\Delta_{lh}
    +\frac{\Delta_{hh}}{\Delta^2}\sum_{l \neq j} G_{lj}\Delta_{(lj; hj)}. 
\end{align}
Here, note that
\begin{align*}
    \sum_{l \neq j} G_{lj}\Delta_{(lj; hj)}
    = \frac{-1}{s}\Delta_{jh}. 
\end{align*}
Using this and \eqref{eq_cancel.y}, we get
\begin{align*}
    \eqref{eq_h_neq_j_hoge.y} = 0 + \frac{\Delta_{hh}}{\Delta^2}
    \cdot \frac{-1}{s}\Delta_{jh}
    =\frac{-1}{s}\frac{\Delta_{jh}\Delta_{hh}}{\Delta^2}. 
\end{align*}

In the case where $h = j$, using Lemma \ref{lem_sum_in_T.y}, we have
\begin{align*}
    \sum_{l}G_{lj}y_j^2\frac{\del}{\del y_j}
    \left( \frac{y_j^2\Delta_{lj}\Delta_{jj}}{\Delta^2}\right)
    &= 2y_j^3 \left(\frac{\Delta_{jj}}{\Delta^2}-y_j^2\frac{\Delta_{jj}^2}{\Delta^3} \right) \sum_l G_{lj} \Delta_{lj} \\
    &= \frac{2y_j^3}{s}\left(\frac{\Delta_{jj}}{\Delta^2}-y_j^2\frac{\Delta_{jj}^2}{\Delta^3} \right), 
\end{align*}
where the last equality uses \eqref{eq_cancel.y}. 
Combining these, we get 
\begin{align}
    T_{jj} &= \frac{1}{4s}\left( 
    -\frac{1}{s}\sum_{1 \le h \le m, h \neq j}y_j^2y_h^2\frac{\Delta_{jh}\Delta_{hh}}{\Delta^2}
    + \frac{2y_j^3}{s}\left(\frac{\Delta_{jj}}{\Delta^2}-y_j^2\frac{\Delta_{jj}^2}{\Delta^3} \right) \right)\notag\\
    &=-\frac{1}{4s^2} \sum_{1 \le h \le m, h\neq j}\left( y_j^2y_h^2 \frac{\Delta_{jh}\Delta_{hh}}{\Delta^2} \right)
    +\frac{y_j^3\Delta_{jj}}{2s^2\Delta} - \frac{y_j^4 \Delta_{jj}^2}{2s^2\Delta^2}. 
    \label{eq_diag.h}
\end{align}

\end{proof}

\begin{rem}\label{rem_minor.y}
The right equality of \eqref{eq_off_diag.h} can be seen by the following general fact in linear algebra. 
\begin{fact}\label{fact_minor.y}
Let $A \in M_n(\R)$ be an invertible $n \times n$-matrix. 
Assume we are given an index set $I \subset \{1, \cdots, n\}$ and we denote its complement by $I' := \{1, \cdots, n\}\setminus I$. 
Let us denote by $[A]_I$ (resp. $[A^{-1}]_{I'}$) the determinant of the submatrix of $A$ (resp. $A^{-1}$) formed by choosing the rows and columns of the index set $I$ (resp. $I'$). 
Then we have
\begin{align*}
    [A]_I = \det(A) \cdot [A^{-1}]_{I'}. 
\end{align*}
\end{fact}
\begin{proof}
Let us list the indices as $I = \{i_1, \cdots, i_k\}$ and $I' = \{j_1, \cdots, j_{n-k}\}$. 
If we denote the standard basis of $\R^n$ by $\{e_i\}_{i = 1}^n$, we have
\begin{align*}
    [A]_I e_{i_1} \wedge \cdots \wedge e_{i_k} \wedge e_{j_1}\wedge \cdots \wedge e_{j_{n-k}}
    &=
    Ae_{i_1} \wedge \cdots \wedge Ae_{i_k} \wedge e_{j_1}\wedge \cdots \wedge e_{j_{n-k}} \\
    &= \det (A) e_{i_1}\wedge \cdots \wedge e_{i_k} \wedge A^{-1}e_{j_1} \wedge \cdots \wedge A^{-1}e_{j_{n-k}} \\
    &= \det(A) \cdot [A^{-1}]_{I'}e_{i_1} \wedge \cdots \wedge e_{i_k} \wedge e_{j_1}\wedge \cdots \wedge e_{j_{n-k}}. 
\end{align*}
\end{proof}
To get \eqref{eq_off_diag.h}, we just apply Fact \ref{fact_minor.y} to the matrix $(Y + A)$ and $I = \{1, \cdots, n\} \setminus\{i, j\}$. 
\end{rem}

\subsection{Estimates of the Ricci curvature}\label{subsec_est_ric.y}
We continue with the ``simplified settings" of the last subsection, where we consider $X = \C^{m} \times \R^{n-m} \times \T^{n-m}$ equipped with the metric $g_s$ given by $G_s = s^{-1}(Y_m + A)$ for a constant positive definite matrix $A \in M_n(\R)$. 
Let us denote by $H \subset X$ the inverse image by $\mu$ of the union of codimension two faces of the polytope, i.e., 
\begin{align}\label{eq_def_H.y}
    H := \cup_{1 \le i \neq j \le m}\mu^{-1}\left( \{(x_1, \cdots, x_n) \in (\R_{\ge 0})^{m} \times \R^{n-m} \ | \ x_i = x_j = 0\} \right)\subset X. 
\end{align}
In this setting, we show the following. 
\begin{prop}\label{prop_lower_bdd_ric.y}
For all $\tilde{r} > 0$, there exists $\kappa \in \R$ such that for all $0<s<1 $, we have
\begin{align*}
    \mathrm{Ric}(g_s) \ge \kappa g_s \mbox{ on } X \setminus B_{g_s}(H, \tilde{r}) .
\end{align*}
Here $B_{g_s}(H, \tilde{r}) := \{x \in X \ | \ d_{g_s}(H, x) < \tilde{r}\}$. 
\end{prop}

\begin{proof}
We denote $z_i = y_i / \sqrt{s}$ for $1 \le i \le n$. 
Recall that $z_i$'s are the coordinates which extends smoothly to the limit space. 
For $r > 0$, $s > 0$, $1 \le i \le m$, we define
\begin{align*}
    X_{s, r, i} := \{(x_1, \cdots, \theta_1, \cdots) \in X \ | \ 
    \sqrt{s}/(2x_j) = z_j \le r \ \forall j = 1, \cdots, i-1, i+1, \cdots, m\}. 
\end{align*}
For any $\tilde{r} > 0$, there exists $r>0$ such that
\begin{align*}
    X \setminus B_{g_s}(H, \tilde{r}) \subset \cup_{i = 1}^{m} X_{s, r, i}. 
\end{align*}
Thus it is enough to show the following. 
\begin{itemize}
\setlength{\parskip}{0cm}
\setlength{\itemsep}{0cm}
 \item[(A)] 
For any $r > 0$ and $1 \le i \le m$, there exists $\kappa \in \R$ such that, for all $0 < s < 1$, we have
\begin{align*}
    \mathrm{Ric}(g_s) \ge \kappa g_s \mbox{ on } X_{s,r,i}. 
\end{align*}
\end{itemize}
Recall that the condition $\mathrm{Ric}(g_s) \ge \kappa g_s $ is equivalent to the condition $T \ge \kappa G$, where $T$ is computed in Proposition \ref{prop_T.y}. 
Note that $T_{ji} \neq 0$ only when $1 \le i, j, \le m$. 
Since $A \in M_n(\R)$ is positive definite and we have
\begin{align*}
    G = \frac{1}{\sqrt{s}}\begin{pmatrix}
z_1& & & & \\
&\ddots & & & \\
& &z_m& & \\
& & & 0 & \\
& & & & \ddots 
\end{pmatrix}
+\frac{1}{s}A, 
\end{align*}
the statement (A) follows from the following statement (B). 
\begin{itemize}
\setlength{\parskip}{0cm}
\setlength{\itemsep}{0cm}
 \item[(B)] 
For any $r > 0$ and $1 \le i \le m$, there exist $C_1 > 0$ and $C_2 \in \R$ such that, for all $0 < s < 1$, we have
\begin{align*}
    s|T_{jl}| &\le C_1 &  (1 \le j \neq l\le m) \\
    T_{jj} &\ge C_2 \max \left\{ \frac{z_j}{\sqrt{s}}, \frac{1}{s}\right\} &  (1 \le j \le m)
    \end{align*}
on $X_{s, r, i}$. 
\end{itemize}
From now on, we show (B). 
We may set $i = 1$. 
We fix $r > 0$. 
We can easily show the following lemma. 

\begin{lem}\label{lem_Delta_est.y}
\begin{enumerate}
    \item We have
    \begin{align*}
        \frac{\Delta_{11}}{\Delta}\le \frac{1}{\sqrt{s}z_1}. 
    \end{align*}
    \item There exists a constant $M > 0$ which only depends on $r$ such that, for all $0 < s < 1$, we have
    \begin{align*}
        \left|\frac{\Delta_{ij}}{\Delta} \right| &\le M & (2 \le \forall i, \forall j \le n) \\
        \left| \frac{\Delta_{1j}}{\Delta}\right| &\le M \min \left\{ 1, \frac{1}{\sqrt{s}z_1}\right\} & (2 \le \forall j \le n)
    \end{align*}
    on $X_{s, r, 1}$. 
\end{enumerate}

\end{lem}
\noindent{\underline{Estimates of $s T_{1j}$, ($j \neq 1$)}}

By Proposition \ref{prop_T.y} and Lemma \ref{lem_Delta_est.y}, we have, for $2 \le j\le k$.
\begin{align}\label{ineq_T_1j.y}
    s|T_{1j}| = \frac{s}{4}z_1^2z_j^2 \frac{\Delta_{1j}^2}{\Delta^2}
    \le \frac{s}{4}z_1^2r^2\frac{M^2}{sz_1^2} 
    =\frac{r^2M^2}{4}
\end{align}
on $X_{s, r, 1}$. 

\noindent{\underline{Estimates of $sT_{jl}$ ($2 \le j \neq l \le m$)}}

By Proposition \ref{prop_T.y} and Lemma \ref{lem_Delta_est.y}, we have, for $2 \le j \neq l \le m$, 
\begin{align}\label{ineq_T_jl.y}
    s|T_{jl}| = \frac{s}{4}z_j^2z_l^2 \frac{\Delta_{jl}^2}{\Delta^2}
    \le \frac{s}{4}r^4M^2
\end{align}
on $X_{s,r,1}$. 

\noindent{\underline{Estimates of $T_{11}$}}

By Proposition \ref{prop_T.y}, we have
\begin{align*}
    T_{11} \ge -\frac{1}{4s^2}\sum_{h = 2}^m y_1^2y_j^2 \frac{\Delta_{1h}\Delta_{hh}}{\Delta^2}. 
\end{align*}
For each $2 \le h\le m$, by Lemma \ref{lem_Delta_est.y} we have
\begin{align*}
    \frac{1}{4s^2}y_1^2y_h^2 \left| \frac{\Delta_{1h}\Delta_{hh}}{\Delta^2}\right|
    &=\frac{1}{4}z_1^2z_h^2\left|\frac{\Delta_{1h}}{\Delta} \frac{\Delta_{hh}}{\Delta} \right|
    \le \frac{1}{4}z_1^2r^2M \left|\frac{\Delta_{1h}}{\Delta}  \right| \\
    &\le \frac{1}{4}z_1^2r^2M^2\min \left\{ 1, \frac{1}{\sqrt{s}z_1}\right\}\\
    & \le \frac{1}{4}r^2M^2\frac{z_1}{\sqrt{s}}. 
\end{align*}
Thus we get 
\begin{align}\label{ineq_T_11.y}
    T_{11} \ge -\frac{m-1}{4}r^2M^2 \frac{z_1}{\sqrt{s}}
\end{align}
on $X_{s,r,1}$. 

\noindent{\underline{Estimates on $T_{jj}$, ($2 \le j\le m$)}}

By Proposition \ref{prop_T.y}, we have
\begin{align*}
    T_{jj} \ge -\frac{1}{4s^2}\sum_{1 \le h \le m, h \neq j} y_j^2y_h^2 \frac{\Delta_{jh}\Delta_{hh}}{\Delta^2}. 
\end{align*}
The term $h = 1$ is estimated as, using Lemma \ref{lem_Delta_est.y}, 
\begin{align*}
    \frac{1}{4s^2}y_j^2y_1^2 \left| \frac{\Delta_{j1}\Delta_{11}}{\Delta^2}\right|
    &=\frac{1}{4}z_j^2z_1^2 \left|\frac{\Delta_{j1}}{\Delta} \right|\frac{\Delta_{11}}{\Delta} \\
    &\le \frac{1}{4}r^2z_1^2M \frac{1}{\sqrt{s}z_1}\frac{1}{\sqrt{s}z_1} \\
    &= \frac{1}{4s}r^2M. 
\end{align*}
The term $2 \le h \le m$, $h \neq j$ is estimated as, using Lemma \ref{lem_Delta_est.y}, 
\begin{align*}
\frac{1}{4s^2}y_j^2y_h^2 \left| \frac{\Delta_{jh}\Delta_{hh}}{\Delta^2}\right|
    =\frac{1}{4}z_j^2z_h^2 \left|\frac{\Delta_{jh}}{\Delta} \right|\frac{\Delta_{hh}}{\Delta}
    \le \frac{1}{4}r^4M^2 . 
\end{align*}
Thus we get
\begin{align}\label{ineq_T_jj.y}
    T_{jj} \ge -\frac{1}{s}\left(\frac{1}{4}r^2M + \frac{m-2}{4}r^4M^2 \right)
\end{align}
on $X_{s,r,1}$. 

Combining \eqref{ineq_T_1j.y}, \eqref{ineq_T_jl.y}, \eqref{ineq_T_11.y} and \eqref{ineq_T_jj.y}, we get the statement (B) and the proof is complete.
\end{proof}

\subsection{Strong spectral convergence}\label{subsec_str_conv.y}

Let us return to the settings in Section \ref{Ricci.h}. 
We set, for $s > 0$ and $b \in B_k$, 
\begin{align*}
    H_s &:= L^2(S, s^{-n/2}\nu_{\hat{g}_s})\otimes \C \\
    H_\infty^b &:= L^2(S_\infty^b, \nu_\infty^b) \otimes \C, 
\end{align*}
where $(S_\infty^b, g_{\infty}^b, \nu_\infty^b, p_\infty^b)$ is the $S^1$-equivariant pointed measured Gromov-Hausdorff limit appearing in Proposition \ref{prop_lim_kBS.y}. 
The goal of this subsection is to prove the following.  

\begin{prop}\label{prop_str_conv.y}
Under the $S^1$-equivariant pointed measured Gromov-Hausdorff convergence given in Proposition \ref{prop_lim_kBS.y}, we have a strong spectral convergence, 
\begin{align*}
    \left(H_s^{\rho_k}, \Delta_{\hat{g}_s}^{\rho_k} \right)
    \to \bigoplus_{b\in B_k}\left((H_\infty^b)^{\rho_k}, (\Delta_\infty^b)^{\rho_k} \right)
\end{align*}
as $s \to 0$. 
\end{prop}

In order to prove Proposition \ref{prop_str_conv.y}, we first consider the ``simplified setting" in the last subsections, and prove the strong convergence in that setting. 
We use the following notations. 
\begin{enumerate}
    \item For integers $0 \le m \le n$ and a positive definite matrix $A \in M_n(\R)$, let us consider the complete Riemannian manifold $(\C^m \times \R^{n-m} \times \T^{n-m}, h_{s, A})$ where the metric $h_{s,A}$ is given by
    \begin{align}\label{eq_model_met.y}
        h_{s,A} :={}^t\!dxG_s dx + {}^t\!d\theta G_s^{-1}d\theta, \quad 
        G_s = \frac{1}{2}X_m^{-1} + s^{-1}A. 
    \end{align}
    \item Let $S_m := \C^m \times \R^{n-m} \times \T^{n-m} \times S^1$ be the frame bundle of the prequantizing line bundle of the model space, and let $\hat{h}_{s, A}$ be the complete Riemannian metric induced by $h_{s,A}$ and the connection on the prequantizing line bundle for which the point $0 \in \C^m \times \R^{n-m} \times \T^{n-m}$ is a strict $l$-Bohr-Sommerfeld point. 
    
\end{enumerate}
By the same argument as Proposition \ref{prop_lim_kBS.y}, we have the $S^1$-equivariant measured Gromov-Hausdorff convergence, 
\begin{align*}
    \left(S_m, \hat{h}_{s, A}, s^{-n/2}\nu_{\hat{h}_{s, A}}, (0, 1)\right)
    \to \left(C_m(A) \times S^1, g_{l, \infty}, \det(A)^{-1/2}d\xi dt, (0, 1)\right). 
\end{align*}

\begin{prop}\label{prop_str_conv_model.y}
Let us fix an integer $0 \le m \le n$ and a positive definite matrix $A \in M_n(\R)$. 
Then we have a strong convergence of spectral structures, 
\begin{align*}
    \left(L^2(S_m, s^{-n/2}\nu_{\hat{h}_{s,A}}), \Delta_{\hat{h}_{s,A}} \right)
    \to \left(L^2(C_m(A) \times S^1, \det(A)^{-1/2}d\xi dt), \Delta_{l, m, A} \right),
\end{align*}
as $s \to 0$. 
\end{prop}

\begin{rem}
In fact, we do not explicitly use Proposition \ref{prop_str_conv_model.y} in the proof of Proposition \ref{prop_str_conv.y}. 
However, the proof of Proposition \ref{prop_str_conv.y} essentially given by reducing the argument to the convergence in this "simplified setting". 
We decided to give a proof of Proposition \ref{prop_str_conv_model.y} here, because it would make clearer what we are doing in the complicated proof of Proposition \ref{prop_str_conv.y}. 
\end{rem}
For the proof of Proposition \ref{prop_str_conv_model.y}, we use the lower boundedness of Ricci curvatures outside the codimension-two faces of $S_m$ given in Proposition \ref{prop_lower_bdd_ric.y}. 
In order to use this property, the following fact is important. 

\begin{lem}\label{lem_sob_cap.y}
Let $F \subset C_m(A) \times S^1$ be the union of codimension-two faces of $C_m(A) \times S^1$. 
Then $C_c^\infty(C_m(A)\times S^1\setminus F )$ is dense in $H^{1,2}(C_m(A)\times S^1, g_{l, \infty}, \det(A)^{-1/2}dtd\xi)$. 
\end{lem}
\begin{proof}
This is standard, shown in exactly the same way as the proof that a codimension-two closed submanifold of a Riemannian manifold has zero Sobolev capacity. 
We recall this argument briefly.

For simplicity, we assume that $A = I_n$, the identity matrix in $M_n(\R)$. 
For $0 < \epsilon<1$, consider 
the Lipschitz function ${\phi}_\epsilon$ on $C_m(I_n) \times S^1 = (\R_{\ge 0})^m \times \R^{n-m} \times S^1$ (with respect to the metric $g_{l, \infty}$) defined by
\begin{align*}
    {\phi}_\epsilon &= \Pi_{1 \le i < j \le m}{\phi}_{\epsilon}^{i,j} \\
    {\phi}_{\epsilon}^{i, j} &:=\begin{cases}
    0 & ( \xi_i^2 + \xi_j^2 \le \epsilon^2) \\
    \frac{\log (\xi_i^2 + \xi_j^2)-2\log \epsilon}{-\log \epsilon} & (\epsilon^2 \le \xi_i^2 + \xi_j^2 \le \epsilon )\\
    1 & (\epsilon \le \xi_i^2 + \xi_j^2 ). 
    \end{cases}
\end{align*}
Then, it is easy to see that, for any function $f \in C_c^\infty(C_m(I_n) \times S^1)$, we have $\phi_\epsilon f \in H^{1,2}(C_m(I_n) \times S^1,g_{l, \infty}, d\xi dt )$, $\mathrm{supp}(\phi_\epsilon f) \in C_m(I_n) \times S^1 \setminus B(F, \epsilon)$ and
\begin{align*}
    \lim_{\epsilon \to 0}\|f - \phi_\epsilon f \|_{H^{1,2}} = 0. 
\end{align*}
It is obvious that we can modify the approximation family $\{\phi_\epsilon f\}_{\epsilon}$ by another family $\{f_\epsilon\}_\epsilon$ with $f_\epsilon \in C^\infty_c (C_m(I_n) \times S^1)$ and $\lim_{\epsilon \to 0}\|f - f_\epsilon \|_{H^{1,2}} = 0$. 
Since $C_c^\infty(C_m(I_n) \times S^1)$ is dense in $H^{1,2}(C_m(I_n)\times S^1, g_{l, \infty}, dtd\xi)$, we get the result in the case $A = I_n$. 
For general $A$, we can just translate the above family $\{\phi_\epsilon\}_\epsilon$ by the linear map $A^{1/2}$ and the result follows by the same argument. 
\end{proof}

\begin{proof}[Proof of Proposition \ref{prop_str_conv_model.y}]
In order to show the strong spectral convergence, by Definition \ref{def_quad_conv.y} we have to check the following two conditions. 
\begin{itemize}
\setlength{\parskip}{0cm}
\setlength{\itemsep}{0cm}
 \item[(S1)] $\|df_\infty\|_{L^2}
\le \liminf_{s \to 0}\|df_s\|_{L^2}$ for any family $\{f_s \in H^{1,2}(S_m, \hat{h}_{s,A}, s^{-n/2}\nu_{\hat{h}_{s, A}})\}_{s}$ and $f_\infty \in L^2(C_m(A)\times S^1, \det(A)^{-1/2}d\xi dt)$
with $f_s \to f_\infty$ $L^2$-weakly, and
 \item[(S2)] for any $f_\infty\in H^{1,2}(C_m(A)\times S^1, g_{l,\infty}, \det(A)^{-1/2}d\xi dt)$ there exists a family $\{ f_s \in H^{1,2}(S_m, \hat{h}_{s,A}, s^{-n/2}\nu_{\hat{h}_{s, A}})\}_s$ 
strongly converging to $f_\infty$ in $L^2$ such that 
$\|df_\infty\|_{L^2} = \lim_{s \to 0} \|df_s\|_{L^2}$. 
\end{itemize}
Recall that we have defined the subset $H \subset \C^m \times \R^{n-m} \times \T^{n-m}$ in \eqref{eq_def_H.y}. 
Let us denote $K := S^1 \times H \subset S_m$. 
Note that approximation maps defined in \eqref{eq_approx_map.y} (and corresponding maps for strict $l$-Bohr-Sommerfeld point for general $l$) send $K$ to $F$. 
By Proposition \ref{prop_lower_bdd_ric.y}, for any $\tilde{r}>0$, there exists $\kappa \in \R$ such that, for all $0 < s < 1$, we have
\begin{align}\label{eq_ric_bound.y}
    \mathrm{Ric}(\hat{h}_{s,A}) \ge \kappa \hat{h}_{s, A} \mbox{ on }
    S_m \setminus B_{\hat{h}_{s, A}}(K, \tilde{r}). 
\end{align}
First we show the condition (S1). 
Assume we are given a family $\{f_s\}_{s > 0}$ and $f_\infty$ as in the assumption in (S1). 
By Vitali's covering theorem, there exist a countable subset $\{p(i)\}_{i \in \N} \subset C_m(A) \times S^1$ and a sequence of positive numbers $\{r(i)\}_{i \in \N}$ such that
\begin{align*}
    B_{g_{l, \infty}}(p(i), r(i)) \cap F = \phi& \mbox{ for all }i, \\
    B_{g_{l, \infty}}(p(i), r(i)) \cap B_{g_{l, \infty}}(p(j), r(j)) = \phi &\mbox{ for }i \neq j, \\
    \nu_\infty(C_m(A)\times S^1\setminus \cup_{i} B_{g_{l, \infty}}(p(i), r(i))) = 0.& 
\end{align*}
For each $i$, let us take a family of points $\{p_s(i) \}_{s > 0} \subset S_m$ such that $p_s(i) \to p(i)$ under the approximation maps. 
For each $i$, there exist $s_0 > 0$ and $\tilde{r} > 0$ such that, for all $0 < s < s_0$, we have
\begin{align*}
    B_{\hat{h}_{s, A}}(p_s(i), r(i)) \subset S_m \setminus B_{\hat{h}_{s, A}}(K, \tilde{r}). 
\end{align*}
We have $f_s|_{B_{\hat{h}_{s, A}}(p_s(i), r(i)) } \in H^{1,2}(B_{\hat{h}_{s, A}}(p_s(i), r(i)) , \hat{h}_{s,A}, s^{-n/2}\nu_{\hat{h}_{s, A}})$ and $f_\infty|_{B_{g_{l, \infty}}(p(i), r(i))} \in L^2(B_{g_{l, \infty}}(p(i), r(i)), \det(A)^{-1/2}dydt )$, as well as the $L^2$-weak convergence $f_s|_{B_{\hat{h}_{s, A}}(p_s(i), r(i))} \to f_\infty|_{B_{g_{l, \infty}}(p(i), r(i))}$. 
From the lower-boundedness of the Ricci curvatures on $S_m \setminus B_{\hat{h}_{s, A}}(K, \tilde{r})$ in \eqref{eq_ric_bound.y}, by \cite[Corollary 4.5]{Honda2015}, we see that
\begin{align*}
    \|df_\infty|_{B_{g_{l, \infty}}(p(i), r(i))}\|_{L^2} \le \liminf_{s \to 0}\|df_s|_{B_{\hat{h}_{s, A}}(p_s(i), r(i))}\|_{L^2}. 
\end{align*}
For each $N \in \N$, there exist $s_1 > 0$ such that, for all $0 < s< s_1$, the balls $\{B_{\hat{h}_{s, A}}(p_s(i), r(i))\}_{i = 1}^N$ are disjoint. 
Thus we get
\begin{align*}
    \sum_{i = 1}^N\|df_\infty|_{B_{g_{l, \infty}}(p(i), r(i))}\|^2_{L^2} \le \liminf_{s \to 0}\|df_s\|^2_{L^2}. 
\end{align*}
Letting $N \to \infty$, we get
\begin{align*}
    \|df_\infty\|_{L^2} \le \liminf_{s \to 0}\|df_s\|_{L^2}. 
\end{align*}

Next we show (S2). 
For a positive number $r>0$, we denote
\begin{align}\label{eq_def_U.y}
    U_r & := C_m(A) \times S^1 \setminus B_{g_{l, \infty}}(F, r), \\
    U_{s, r} &:= S_m \setminus B_{\hat{h}_{s, A}}(K, {r}). 
\end{align}
The measured Gromov-Hausdorff convergence restricts to that of subspaces $U_{s, r} \to U_r$ for any $r>0$. 
By Lemma \ref{lem_sob_cap.y}, it is enough to show (S2) when $\mathrm{supp}(f_\infty)$ is compact and contained in $C_m(A) \times S^1 \setminus F$. 
In this case, there exists a positive number $r > 0$ with $\mathrm{supp}(f_\infty) \subset U_{r}$. 
By lower-boundedness of Ricci curvatures on $U_{r, s}$ in \eqref{eq_ric_bound.y}, such a sequence $f_s \in H^{1,2}(U_{r,s}, \hat{h}_{s,A}, s^{-n/2}\nu_{\hat{h}_{s,A}})$ exists by \cite[Theorem 4.2]{Honda2011}. 
\end{proof}

Now we return to the original settings. 
The following lemma is the essential part of the proof of Proposition \ref{prop_str_conv.y}. 

\begin{lem}\label{lem_str_conv_ptwise.y}
Let $b \in P\cap \frac{\Z^n}{l}$ be a strict $l$-Bohr-Sommerfeld point.  
Under the pointed measured Gromov-Hausdorff convergence given in Proposition \ref{prop_lim_kBS.y}, we have a strong spectral convergence, 
\begin{align*}
    \left(H_s, \Delta_{\hat{g}_s} \right)
    \to \left(H_\infty^b, \Delta_\infty^b \right)
\end{align*}
as $s \to 0$. 
\end{lem}

\begin{proof}
Take an action-angle coordinate around $b$ as in Section \ref{sec_lim_sp.y}, so that $b = 0\in W \subset \R_{\ge 0}^m \times \R^{n-m}$ (where we allow the translation of action coordinate of the form $x \mapsto Ax + c$, where $A \in GL_n\Z$ and $c \in \frac{\Z^n}{l}$). 
The metric $g_s$ on $\tilde{W} := \mu_P^{-1}(W)$ is given by
\begin{align*}
    g_s = {}^t\!dx G_s dx + {}^t\!d\theta G_s^{-1}d\theta, \quad
    G_s = \frac{1}{2} X_m^{-1} + s^{-1}A + B
\end{align*}
for some $A, B \in C^\infty(W) \otimes M_n(\R)$, where $A$ takes values in positive definite matrices. 
The model space at $b$ is $(\C^m \times \R^{n-m} \times \T^{n-m}, h_{s, A(0)})$, where the metric $h_{s, A(0)}$ is given by \eqref{eq_model_met.y}. 
We consider the prequantum line bundle on the model space which coincides with that on $\tilde{W}$ under the above inclusion, and denote the frame bundle with induced metric by $(S_m = S^1 \times \C^m \times \R^{n-m} \times \T^{n-m}, \hat{h}_{s, A(0)})$. 
To simplify the notations, we set $\hat{h}_s := \hat{h}_{s, A(0)}$ in this proof. 
We regard $S|_{\tilde{W}}$ as a subset of $S_m$ by the above inclusion $W \hookrightarrow \R^m_{\ge 0} \times \R^{n-m}$, and consider two families of metrics $\{\hat{h}_{s}\}_s$ and $\{\hat{g}_s\}_s$ on $S|_{\tilde{W}s}$. 

Since the metric is expanding in the base direction and the matrices $A$ and $ B$ are smooth up to the faces of $W$, we see that, for any $R > 0$ and $\epsilon > 0$, there exists $s_{\epsilon, R} > 0$ such that, for any $0 < s < s_{\epsilon, R}$, we have
\begin{align}\label{ineq_met_equiv.y}
    (1 - \epsilon) \hat{h}_s \le \hat{g}_s \le (1 + \epsilon) \hat{h}_s \mbox{ on }B_{\hat{g}_s}(\mu_P^{-1}(b), R). 
\end{align}

In order to show the strong spectral convergence, we have to check the following two conditions. 
\begin{itemize}
\setlength{\parskip}{0cm}
\setlength{\itemsep}{0cm}
 \item[(S1)] $\|df_\infty\|_{L^2}
\le \liminf_{s \to 0}\|df_s\|_{L^2}$ for any $\{f_s \in H^{1,2}(S, \hat{g}_s, s^{-n/2}\nu_{\hat{g}_s})\}_{s}$ and $f_\infty \in L^2(S_\infty^b, \nu_\infty^b)$
with $f_s \to f_\infty$ $L^2$-weakly, and
 \item[(S2)] for any $f_\infty\in H^{1,2}(S_\infty^b, g_\infty^b, \nu_\infty^b)$ there exists a sequence $\{ f_s \in H^{1,2}(S, \hat{g}_s, s^{-n/2}\nu_{\hat{g}_s})\}_s$ 
strongly converging to $f_\infty$ in $L^2$ such that 
$\|df_\infty\|_{L^2} = \lim_{s \to 0} \|df_s\|_{L^2}$. 
\end{itemize}
Both conditions can be shown by the corresponding results for the model metric $\hat{h}_s$ by \eqref{ineq_met_equiv.y}, as follows.

First we show (S1). 
Assume we are given $\{f_s\}_s$ and $f_\infty$ as in the assumption of (S1). 
By Vitali's covering theorem, there exist a countable subset $\{p(i)\}_{i \in \N} \subset S_\infty^b$ and a sequence of positive numbers $\{r(i)\}_{i \in \N}$ such that
\begin{align*}
    B_{g_\infty^b}(p(i), r(i)) \cap F = \phi& \mbox{ for all }i, \\
    B_{g_\infty^b}(p(i), r(i)) \cap B_{g_\infty^b}(p(j), r(j)) = \phi &\mbox{ for }i \neq j, \\
    \nu_\infty(S_\infty^b\setminus \cup_{i} B_{g_{\infty}^b}(p(i), r(i))) = 0.& 
\end{align*}
For each $i$, let us take a family of points $\{p_s(i) \}_{s > 0} \subset S$ such that $p_s(i) \to p(i)$ under the approximation maps giving the Gromov-Hausdorff convergence. 
For each $i$, there exist $s_0 > 0$ and $\tilde{r} > 0$ such that, for all $0 < s < s_0$, we have
\begin{align*}
    B_{\hat{h}_s}(p_s(i), r(i)) \subset S|_{\tilde{W}} \cap (S_m \setminus B_{\hat{h}_s}(K, \tilde{r})). 
\end{align*}
From now on, only in this proof, we use the following notations. 
\begin{align*}
    B_\infty(i) &:= B_{g_\infty^b}(p(i), r(i)) \\
    B_s(i) &:= B_{\hat{h}_s}(p_s(i), r(i))
\end{align*}
We have $f_s|_{B_s(i) } \in H^{1,2}(B_s(i) , \hat{g}_{s}, s^{-n/2}\nu_{\hat{g}_s})$ and $f_\infty|_{B_\infty(i)} \in L^2(B_\infty(i), \nu_\infty^b )$, as well as the $L^2$-weak convergence $f_s|_{B_s(i)} \in L^2(B_s(i), s^{-n/2}\nu_{\hat{g}_s}) \to f_\infty|_{B_\infty(i)}\in L^2(B_\infty(i), \nu_\infty^b)$. 
By \eqref{ineq_met_equiv.y}, we also see that for $0 < s < s_0$ we have $f_s|_{B_s(i) } \in H^{1,2}(B_s(i) , \hat{h}_s, s^{-n/2}\nu_{\hat{h}_s})$ and
we have the $L^2$-weak convergence
$f_s|_{B_s(i)} \in L^2(B_s(i), s^{-n/2}\nu_{\hat{h}_s}) \to f_\infty|_{B_\infty(i)}\in L^2(B_\infty(i),\nu_\infty^b)$.
From the lower-boundedness of the Ricci curvatures on $S_m \setminus B_{\hat{h}_{s}}(K, \tilde{r})$ in \eqref{eq_ric_bound.y}, by \cite[Corollary 4.5]{Honda2015}, we see that
\begin{align*}
    \|df_\infty|_{B_\infty(i)}\|_{L^2(S_\infty^b, g_\infty^b, \nu_\infty^b)} \le \liminf_{s \to 0}\|df_s|_{B_s(i)}\|_{L^2(S_m,\hat{h}_s,  s^{-n/2}\nu_{\hat{h}_s})}, 
\end{align*}
where we denoted by $\|df_\infty\|_{L^2(S_\infty^b, g_\infty^b, \nu_\infty^b)}$ the $L^2$-norm, with respect to the measure $\nu_\infty^b$, of the function $|df_\infty|_{g_\infty^b}$, etc. 
By \eqref{ineq_met_equiv.y}, we have
\begin{align*}
    \liminf_{s \to 0}\|df_s|_{B_s(i)}\|_{L^2(S_m,\hat{h}_s, s^{-n/2}\nu_{\hat{h}_s})}
    =\liminf_{s \to 0}\|df_s|_{B_s(i)}\|_{L^2(S,\hat{g}_s,  s^{-n/2}\nu_{\hat{g}_s})}. 
\end{align*}
Thus we get
\begin{align*}
    \|df_\infty|_{B_\infty(i)}\|_{L^2(S_\infty^b, g_\infty^b, \nu_\infty^b)} \le
    \liminf_{s \to 0}\|df_s|_{B_s(i)}\|_{L^2(S,\hat{g}_s,  s^{-n/2}\nu_{\hat{g}_s})}.
\end{align*}
For each $N \in \N$, there exist $s_1 > 0$ such that, for all $0 < s< s_1$, the balls $\{B_s(i)\}_{i = 1}^N$ are disjoint. 
Thus we get
\begin{align*}
    \sum_{i = 1}^N\|df_\infty|_{B_{\infty}(i)}\|_{L^2(S_\infty^b, g_\infty^b, \nu_\infty^b)}^2 \le \liminf_{s \to 0}\|df_s\|_{L^2(S,\hat{g}_s,  s^{-n/2}\nu_{\hat{g}_s})}^2. 
\end{align*}
Letting $N \to \infty$, we get
\begin{align*}
    \|df_\infty\|_{L^2(S_\infty^b, g_\infty^b, \nu_\infty^b)} \le \liminf_{s \to 0}\|df_s\|_{L^2(S,\hat{g}_s, s^{-n/2}\nu_{\hat{g}_s})}. 
\end{align*}
So we get (S1). 

Next we show (S2). 
For $r > 0$ and $R > 0$, let us denote
\begin{align*}
    U_r(R) & := U_r \cap B_{g^b_\infty}(p_\infty^b, R) \subset S_\infty^b = C_m(A(0))\times S^1 \\
    U_{s, r}(R) &:= U_{r,s} \cap B_{\hat{h}_s}(u_b, R)\subset  S_m,  
\end{align*}
where $U_r$ and $U_{s,r}$ are defined in \eqref{eq_def_U.y}. 
Note that, for any $R>0$, for sufficiently small $s>0$ we have $U_{s,r}(R) \subset \tilde{W} \times S^1 \subset S$.

By Lemma \ref{lem_sob_cap.y}, it is enough to show (S2) when $\mathrm{supp}(f_\infty)$ is compact and contained in $S_\infty^b \setminus F =C_m(A(0)) \times S^1 \setminus F$. 
In this case, there exist  positive numbers $r > 0$ and $R>0$ with $\mathrm{supp}(f_\infty) \subset U_{r}(R)$. 
By lower-boundedness of Ricci curvatures of $\hat{h}_s$ on $U_{r, s}$ in \eqref{eq_ric_bound.y}, there exists a family $\{f_s \in H^{1,2}(U_{r,s}(R), \hat{h}_s, s^{-n/2}\nu_{\hat{h}_s})\}_{s}$ such that $f_s \in L^2(U_{s, r}(R), s^{-n/2}\nu_{\hat{h}_s}) \to f_\infty \in L^2(U_r(R), \nu_\infty^b)$ strongly in $L^2$ and $\|df_\infty\|_{L^2(U_{r}(R), g_\infty^b,\nu_\infty^b)} = \lim_{s \to 0} \|df_s\|_{L^2(U_{s, r}(R),\hat{h}_s, s^{-n/2}\nu_{\hat{h}_s})}$ by \cite[Theorem 4.2]{Honda2011}. 
Again using \eqref{ineq_met_equiv.y}, we see that we also have $f_s \in L^2(U_{s, r}(R), s^{-n/2}\nu_{\hat{g}_s}) \to f_\infty \in L^2(U_r(R), \nu_\infty^b)$ strongly in $L^2$. 
Moreover, we also have $\lim_{s \to 0} \|df_s\|_{L^2(U_{s,r}(R), \hat{g}_s, s^{-n/2}\nu_{\hat{g}_s})} = \lim_{s \to 0} \|df_s\|_{L^2(U_{s,r}(R),\hat{h}_s, s^{-n/2}\nu_{\hat{h}_s})}$. 
So we see that the family $\{f_s \in H^{1, 2}(S, \hat{g}_s, s^{-n/2}\nu_{\hat{g}_s})\}_{0 < s < \delta}$ has the desired property. 

\end{proof}

Now we can prove Proposition \ref{prop_str_conv.y}. 
\begin{proof}[Proof of Proposition \ref{prop_str_conv.y}]
With Lemma \ref{lem_str_conv_ptwise.y} at hand, the proof is essentially the same as that of \cite[Proposition 3.14]{HY2019}. 
By Lemma \ref{lem_str_conv_ptwise.y} and \cite[Proposition 3.11]{HY2019}, we have, for each $b \in B_k$, 
\begin{align}\label{eq_str_conv_ptwise.y}
    \left(H_s^{\rho_k}, \Delta_{\hat{g}_s}^{\rho_k} \right)
    \to \left((H_\infty^b)^{\rho_k}, (\Delta_\infty^b)^{\rho_k} \right)
\end{align}
strongly as $s \to 0$. 

Let us denote by $E_s$ the spectral measure of $\left(H_s^{\rho_k}, \Delta_{\hat{g}_s}^{\rho_k} \right)$ and by $E_\infty = \bigoplus_{b \in B_k}E_\infty^b$ that of $\bigoplus_{b \in B_k}\left((H_\infty^b)^{\rho_k}, (\Delta_\infty^b)^{\rho_k} \right)$. 
Take any two real numbers $\lambda< \mu$ which are not in the point spectrum of $\bigoplus_{b \in B_k}(\Delta_{\infty}^b)^{\rho_k}$.
Then we must show that $E_s((\lambda, \mu]) \to E_\infty((\lambda, \mu])$ strongly. 
Take a strongly convergent family $f_s \to f_\infty$, where $f_s\in H_s^{\rho_k}$. 
We must show $E_s((\lambda, \mu])f_s \to E_\infty((\lambda, \mu]) f_\infty$ strongly. 

We decompose $f_\infty = \sum_{b\in B_k} f_\infty^b$ where $u_\infty^b \in (H_\infty^b)^{\rho_k}$. 
We may decompose the family $\{f_s\}_s$ into families $\{f_s^b\}_s$ ($b \in B_k $), where $f_s = \sum_{b\in B_k} f_s^b$ and $f_s^b \to f_\infty^b$ strongly for each $b$. 
By the strong spectral convergence in \eqref{eq_str_conv_ptwise.y}, we see that $E_s((\lambda, \mu])f^b_s \to E^b_\infty((\lambda, \mu]) f_\infty^b$ strongly.
We take a sum over $b$ and get the result. 
\end{proof}

\section{Compact spectral convergence}\label{sec_cpt_conv.y}
Finally, in this section, we prove the main result of this paper, Theorem \ref{thm_main.y}. 
Since we know the strong convergence by Proposition \ref{prop_str_conv.y}, in order to show the compact convergence, what we need to show is the item (4) of Definition \ref{def_quad_conv.y}, i.e., that given any sequence $\{f_i \in (L^2(S; \hat{g}_{J_{s_i}})\otimes \C)^{\rho_k}\}_i$ with a uniform $H^{1,2}$-norm, we can find a strongly convergent subsequence. 
In order for this, what we need to prove is, roughly speaking, that given any such sequence $\{f_i\}_i$, they stay in a certain distance from the set $B_k$ of Bohr-Sommerfeld points of level $k$. 
We refer the reader to the corresponding argument in the case of non-singular Lagrangian fibration in \cite[Section 4]{HY2019}; the idea is essentially the same as the one used there, but here the proof is a little more technical because we do not have the uniform lower bound of the Ricci curvature.  

\subsection{Localization of $H^{1,2}$-bounded sequences to Bohr-Sommerfeld fibers}
In this subsection we show the following. 

\begin{prop}\label{thm_loc_bs_toric. y}
Assume that for each $0<s < \delta$, a function $f_s \in (C^\infty(S) \otimes \C)^{\rho_k}$ is chosen so that $\|f_s\|_{L^2(S, s^{-n/2}\nu_{\hat{g}_s})} = 1$ and $\sup_{0 < s < \delta } \|df_s\|_{L^2(S, s^{-n/2}\nu_{\hat{g}_s})} < \infty$. 
Then for any $\epsilon > 0$, there exists $C > 0$ such that
for all $0 < s <\delta$, we have
\[
\| f_s|_{\mu^{-1}(B(B_k, \sqrt{s}C))} \|^2\geq 1 - \epsilon. 
\]
Here $B(B_k, \sqrt{s}C) = \{x \in P \ | \ \inf_{y \in B_k} \|x-y\| < \sqrt{s}C \}$. 
\end{prop}

\begin{proof}
We apply the results in \cite[Section 4]{HY2019}. 
For a point $b \in \breve{P}$, let us denote by $g_b=g_{b, ij}d\theta^i d\theta^j$ the fiberwise metric on $\mu_P^{-1}(b)$ and we define
\begin{align*}
N_b
&:= \sup_{\theta \in \mathbb{T}^n}\left\{ N_b(\theta)\in\R_{+};\, 
N_b(\theta) \mbox{ is the maximum eigenvalue of }
(g_{b, ij}(\theta))_{i,j}\right\}, \\
\lambda(k,b)
&:= \inf \left\{ \sum_{i=1}^n ( m_i+kb_i)^2;\, m_1,\ldots,m_n\in\Z \right\}.
\end{align*}
There exists a positive constant $M > 0$ such that $N_b \leq sM$ for any $b \in \breve{P}$. 
For each $c > 0$ we denote $B_{s, c} := B(B_k, \sqrt{s}c) \subset P$. 
We have $\lambda(k, b) \ge k^2(\sqrt{s}c)^2$ for all $b \in P \setminus B_{s, c}$. 
Applying \cite[Proposition 4.3]{HY2019}, for all $f \in (C^\infty(S)\otimes \mathbb{C})^{\rho_k}$ we have
\begin{align*}
\int_{S|_{\mu^{-1}(\breve{P} \setminus B_{s, c})}} |df|^2_{\hat{g}_s} d\nu_{\hat{g}_s}
&\geq 2\pi\left(k^2 + \frac{k^2(\sqrt{s}c)^2}{sM}
\right) 
 \int_{S|_{\mu^{-1}(\breve{P} \setminus B_{s, c})}}  |f|^2 d\nu_{\hat{g}_s} \\
&= 2\pi k^2\left(1 + \frac{c^2}{M}
\right) \int_{S|_{\mu^{-1}(\breve{P} \setminus B_{s, c})}}  |f|^2 d\nu_{\hat{g}_s}. 
\end{align*}
Noting that $S|_{\mu^{-1}(P \setminus \breve{P})}$ is of measure zero, we get
\[
\int_{S|_{\mu^{-1}(P \setminus B_{s, c})}} |df|^2_{\hat{g}_s} d\nu_{\hat{g}_s} \geq 2\pi k^2\left(1 + \frac{c^2}{M}
\right) \int_{S|_{\mu^{-1}(P \setminus B_{s, c})}}  |f|^2 d\nu_{\hat{g}_s}.
\] 

Assume we are given a family $\{f_s\}_s$ as in the statement of the proposition. 
By the assumption we can take $\Lambda > 0$ such that $\|df_s\|^2_{L^2} \le \Lambda$ for all $0 < s < \delta$. 
Given any positive number $\epsilon > 0$, we take $C > 0$ so that $k^2(1 + \frac{C^2}{M}) > \frac{\Lambda}{\epsilon}$. 
Then we have
\[
\int_{S|_{\mu^{-1}(P \setminus B_{s, c})}} |f_s|^2d\nu_{\hat{g}_s} \leq \frac{\epsilon}{2\pi \Lambda} \int_{S|_{\mu^{-1}(P \setminus B_{s, c})}} |df_s|_{\hat{g}_s}^2d\nu_{\hat{g}_s} 
\leq \frac{\epsilon}{2\pi}.
\]
So we get
\[
\| f_s\|^2_{\mu^{-1}(B_{s, C})} \geq 1 - \frac{\epsilon}{2\pi}. 
\]
This proves the proposition. 

\end{proof}

\subsection{Convergence of 
$H^{1,2}$-bounded sequences}\label{sec. rellich.h}
In this section, we consider the family $\{J_s\}_{0<s<\delta }$ of $\omega$-compatible complex structures defined in Section \ref{Ricci.h}. 

In this subsection we prove 
the following proposition. 
For ease of notations, we denote the $H^{1,2}(S, \hat{g}_s, s^{-n/2}\nu_{\hat{g}_s})$-norm by $\|\cdot \|_{H^{1,2}(\hat{g}_s)}$. 
Recall that we have 
\begin{align*}
    \| f\|_{H^{1,2}(\hat{g}_s)}^2
= \| f\|_{L^2(S,s^{-n/2}\nu_{\hat{g}_s})}^2 
+ \| df\|_{L^2(S,\hat{g}_s, s^{-n/2}\nu_{\hat{g}_s})}^2
\end{align*}
for $f \in C^\infty(S)$. 
\begin{prop}\label{prop_compact_conv.h}
Let $s_i>0$ and $\lim_{i\to\infty }s_i = 0$. 
Take a sequence $\{f_i\}_{i \in \Z_{>0}}\subset (C^\infty(S)\otimes \C)^{\rho_k}$ 
such that 
\[
\limsup_{i\to\infty}\|f\|_{H^{1,2}(\hat{g}_s)} <\infty.
\]
Then there are a subsequence 
$\{ f_{i(j)}\}_{j=1}^\infty
\subset \{ f_i\}_{i=1}^\infty$ 
and $f_\infty^b\in (L^2(S_\infty^b, \nu_\infty^b)
\otimes\C)^{\rho_k}$ for 
each $b\in B_k$ such that 
$f_{i(j)}\to \left( 
f_\infty^b \right)_b$ strongly as 
$j\to \infty$. 
\end{prop}
The proof of Proposition \ref{prop_compact_conv.h} 
is similar to the proof of \cite[Proposition 4.7]{HY2019}, 
however, we have to add some arguments 
since we do not have the uniform lower bound of 
the Ricci curvatures of $\{ \hat{g}_s\}_s$.

Recall that $\hat{g}_s$ was determined by 
the matrix $G_s=\frac{1}{2}X_m+s^{-1}A+B$ (where the right hand side makes sense only locally). 
For each $b \in B_k$, we define metrics $g'_{s, b}$ and $\hat{g}_{s,b}'$ by 
\begin{align*}
g_{s, b}'
&:= {}^t\!dxG_s' dx + {}^t\!d\theta (G_s')^{-1}d\theta,\\
\hat{g}_{s,b}'
&:= (dt - {}^t\!xd\theta)^2 + {}^t\!dxG_s' dx + {}^t\!d\theta (G_s')^{-1}d\theta,\\
G_s' &:= \frac{1}{2}X_m^{-1} + s^{-1}I_n, 
\end{align*}
where $I_n$ is the identity matrix. 
Here, 
$g_{s,b}'$ is a metric 
defined on an open neighborhood 
\begin{align*}
X_b'\subset X
\end{align*}
of $\mu_P^{-1}(b)$, and
$\hat{g}_{s,b}'$ is defined on its lift, 
\begin{align*}
S_b' := \pi^{-1}(X'_b)\subset S. 
\end{align*}
For this metric, we have the lower bound for the Ricci curvature as follows. 

\begin{lem}\label{lem_ric_pos.y}
For any $r>0$ there is $s_r>0$ such that 
${\rm Ric}_{g_{s,b}'}\ge 0$ holds on $B_{g_{s,b}'}(\mu_P^{-1}(b),r)$.
\end{lem}
\begin{proof}
Denote by $\rho_{s,b}'$ the Ricci form 
of $g_{s,b}'= {}^t\!dxG_s' dx + {}^t\!d\theta (G_s')^{-1}d\theta$ 
and let $T_s':= 4G_s'\rho_s' G_s'$. 
By Proposition \ref{prop_T.y}, 
$T_{ji}' = 0$ for $i \neq j$, 
\begin{align*}
	T_{jj}' &=
	\frac{y_j^3\Delta_{jj}}{2s^2\Delta^2} 
	\left( \Delta - y_j\Delta_{jj}\right)\\
	&=
	\frac{y_j^3\Delta (y_j + 1)^{-1}}{2s^2\Delta^2} 
	\left( \Delta - y_j(y_j + 1)^{-1}\Delta\right)\\
	&=
	\frac{y_j^3 (y_j + 1)^{-1}}{2s^2} 
	\left( 1 - y_j(y_j + 1)^{-1}\right)\\
	&=
	\frac{y_j^3 }{2s^2(y_j + 1)^2}
\end{align*}
for $1\le j\le m$ and 
\begin{align*}
	T_{jj}' &=
	\frac{y_j^3\Delta_{jj}}{2s^2\Delta^2} 
	\left( \Delta - y_j\Delta_{jj}\right)\\
	&=
	\frac{y_j^3\Delta}{2s^2\Delta^2} 
	\left( \Delta - y_j \Delta\right)\\
	&=
	\frac{y_j^3}{2s^2} 
	\left( 1 - y_j \right)
\end{align*}
for $m+1\le j\le n$. 
Then $(G_s')^{-1}T_s'$ is given by 
\begin{align*}
	\{ (G_s')^{-1}T_s'\}_{jj} 
	&=
	\frac{y_j^3 }{2s(y_j + 1)^3}\ge 0
\end{align*}
for $1\le j\le m$ and 
\begin{align*}
	\{ (G_s')^{-1}T_s'\}_{jj}' =
	\frac{y_j^3}{2s} 
	\left( 1 - y_j \right)
	=
	\frac{ s^2 }{16 x_j^3} 
	\left( 1 - \frac{s}{2x_j}\right)
\end{align*}
for $m+1\le j\le n$. 
Here, the value of $x_j$ for $m+1\le j\le n$ 
on $B_{g_{s,b}'}(\mu_P^{-1}(b),r)$ 
is close to $x_j(b)>0$ if $s\to 0$, 
hence we have shown that 
for any $r>0$ there is $s_r>0$ such that 
${\rm Ric}_{g_{s,b}'}\ge 0$ holds on $B_{g_{s,b}'}(\mu_P^{-1}(b),r)$. 
\end{proof}

From now on we put, for $p \in X'_b$, $u \in S'_b$ and $r, \delta >0$, 
\begin{align*}
B(s,p,r)&:= B_{g_s}(p,r),\quad 
B'(s,p,r):= B_{g_{s,b}'}(p,r),\\
\hat{B}(s,u,r)&:= B_{\hat{g}_s}(u,r),\quad 
\hat{B}'(s,u,r):= B_{\hat{g}_{s,b}'}(u,r),\\
B(b,\delta) &:= \{ x\in \R^n;\, \| x-b \|<\delta\}.
\end{align*}
We take sufficiently small $\delta>0$ so that 
\begin{align*}
\del P \cap \overline{ B(b,\delta) } \subset 
\bigcup_{j=1}^m\left\{ x\in \R^n; x_j=0\right\}
\end{align*}
holds, then we may suppose that 
$\hat{g}_{s,b}'$ is defined on 
\begin{align}\label{def_S'b}
S_b':=(\mu_P\circ\pi)^{-1}(B(b,\delta)). 
\end{align}

By Proposition \ref{prop_ball.h} (i), 
there is 
$s_r>0$ for any $r>0$ such that 
\begin{align*}
B(s,p,r) &\subset \mu_P^{-1}(B(b,\delta)),\quad 
B'(s,p,r) \subset \mu_P^{-1}(B(b,\delta)),\\
\hat{B}(s,u,r) &\subset (\mu_P\circ\pi)^{-1}(B(b,\delta)),\quad 
\hat{B}'(s,u,r) \subset (\mu_P\circ\pi)^{-1}(B(b,\delta))
\end{align*}
for any $0<s\le s_r$, $p\in \mu_P^{-1}(b)$ 
and $u\in (\mu_P\circ\pi)^{-1}(b)$. 
\begin{lem}
There are positive constants $C,\delta,s_0>0$ such that 
\begin{align*}
C^{-1}g_{s,b}'|_{\mu_P^{-1}(B(b,\delta))}&\le g_s|_{\mu_P^{-1}(B(b,\delta))} \le Cg_{s,b}'|_{\mu_P^{-1}(B(b,\delta))},\\
C^{-1}\hat{g}_{s,b}'|_{(\mu_P\circ\pi)^{-1}(B(b,\delta))}&\le \hat{g}_s|_{(\mu_P\circ\pi)^{-1}(B(b,\delta))} \le C\hat{g}_{s,b}'|_{(\mu_P\circ\pi)^{-1}(B(b,\delta))},\\
\end{align*}
holds for any $0<s\le s_0$. 
\label{lem_quasi_isom.h}
\end{lem}
\begin{proof}
Let $\delta$ be as above. 
For the positive definite symmetric matrix $K$, 
denote by $\max K$ and $\min K$ the maximum and 
the minimum of the eigenvalues of $K$, respectively. 
Put 
\begin{align*}
C_0&:=\max\left\{ \sup_{x\in \overline{ B(b,\delta) } } \{ \max A(x)\},\ 
\sup_{x\in \overline{ B(b,\delta) } } \{ \max B(x)\},\ 1 \right\}<\infty,\\
C_1
&:=\min\left\{ \inf_{x\in \overline{ B(b,\delta) } } \{ \min A(x)\},\ 1 \right\}>0.
\end{align*}
If $s\le 1$ and $x\in B(b,\delta)$, 
then one can see 
\begin{align*}
G_s \le \frac{1}{2}X_m+s^{-1}A+B 
&\le \frac{1}{2}X_m+s^{-1}C_0 I_n+C_0 I_n\\
&\le \frac{1}{2}X_m+2s^{-1}C_0 I_n
\le 2C_0 G_s'
\end{align*}
and 
\begin{align*}
G_s' \le \frac{1}{2}X_m+s^{-1}I_n 
&\le \frac{1}{2}X_m+s^{-1}C_1^{-1}A(x) 
\le C_1^{-1} G_s.
\end{align*}
Hence we obtain 
\begin{align*}
G_s' \le C G_s, \quad 
G_s \le C G_s', \quad 
(G_s')^{-1} \le C G_s^{-1}, \quad 
G_s^{-1} \le C (G_s')^{-1}, \quad 
\end{align*}
by putting $C=\max\{ 2C_0,C_1^{-1}\}$, 
which implies the assertion. 
\end{proof}
Let $\nu_{s,b}':=s^{-n/2}\nu_{\hat{g}_{s,b}'}$. 
By Lemma \ref{lem_quasi_isom.h}, 
we can see that there are constants 
$s_r>0, C>0$ such that 
\begin{align*}
B(s,p,C^{-1}r) &\subset B'(s,p,r) \subset B(s,p,Cr) \subset \pi(S_b'),\\
\hat{B}(s,u,C^{-1}r) &\subset \hat{B}'(s,u,r) \subset \hat{B}(s,u,Cr)\subset S_b',\\
C^{-1}s^{-n/2}\nu_{\hat{g}_s}(W) &\le \nu_{s, b}'(W) \le Cs^{-n/2}\nu_{\hat{g}_s}(W)
\end{align*}
hold for any $0<s<s_r$ and any Borel subset 
$W\subset S_b'$.

Now we prove 
Proposition \ref{prop_compact_conv.h}.
\begin{proof}[Proof of Proposition \ref{prop_compact_conv.h}]
Denote the $H^{1,2}(S_b', \hat{g}'_{s, b}, \nu'_{s, b})$-norm by $\|\cdot \|_{H^{1,2}(\hat{g}_{s,b}')}$. 
(Recall that $S'_b$ is defined in \eqref{def_S'b}. )
By Lemma \ref{lem_quasi_isom.h}, 
we have 
\begin{align*}
C^{-1}\| f|_{S'_b}\|_{H^{1,2}(\hat{g}_s)}\le
\| f|_{S'_b}\|_{H^{1,2}(\hat{g}_{s,b}')}\le
C\| f|_{S'_b}\|_{H^{1,2}(\hat{g}_s)}.
\end{align*}
Suppose $\limsup_{i\to \infty}\| f_i\|_{H^{1,2}(\hat{g}_{s_i})}^2=a<\infty$ and 
fix $u_b\in(\mu_P\circ\pi)^{-1}(b)$. 
Then we have 
$\limsup_{i\to \infty}\| f_i|_{\hat{B}(s_i,u_b,r)}\|_{H^{1,2}(\hat{g}_{s_i})}^2\le a$ for any $r>0$, 
hence 
\begin{align*}
\limsup_{i\to \infty}\| f_i|_{\hat{B}'(s_i,u_b,r)}\|_{H^{1,2}(\hat{g}_{s_i,b}')}^2\le Ca
\end{align*}
holds for any $r>0$. 

Now fix a point $b \in B_k$. 
Assume that $b$ is a strict 
Bohr-Sommerfeld point of level $l$, 
where $k\in l\Z$, and 
we write  
\begin{align*}
H_{i,b}'&:= \left( L^2(S_b',\nu_{s_i,b}')\otimes\C\right)^{\rho_k},\\
C_m' &:= C_m(I_n)=\R_{\ge 0}^{m}\times\R^{n-m},\\
H_{\infty,b}'&:= \left( L^2(C_m'\times S^1,dtd\xi)\otimes\C\right)^{\rho_k},\\
\mathcal{C}_b'&:=(C_c(C_m'\times S^1)\otimes \C)^{\rho_k}\subset H_{\infty,b}',\\
H_i&:= \left( L^2(S,s_i^{-n/2}\nu_{\hat{g}_{s_i}})\otimes\C\right)^{\rho_k},\\
C_m &:= C_m(A(0)), \\
H_{\infty,b}&:= \left( L^2(C_m\times S^1,\det (A(0))^{-1/2}dtd\xi)\otimes\C\right)^{\rho_k},\\
\mathcal{C}_b&:=(C_c(C_m\times S^1)\otimes \C)^{\rho_k}\subset H_{\infty,b}
\end{align*}
in this proof for the simplicity. 
Now, for any $R > 0$, we have the Ricci curvatures of 
$\{ \hat{g}_{s,b}'\}_s$ on $\{\hat{B}'(s,u_b,R)\}_s$ are positive for all sufficiently small $s>0$ by Lemma \ref{lem_ric_pos.y}. 
Therefore, 
by the similar argument 
to the proof of 
\cite[Proposition 4.7]{HY2019}, 
we have a subsequence 
$\{ i(j)\}_j\subset \{ i\}$ and 
$f_{\infty,b}'\in H_{\infty,b}'$ 
such that 
\begin{align}
f_{i(j)}|_{\hat{B}'(s_{i(j)},u_b,R)} \to 
f_{\infty,b}'|_{\hat{B}'(\infty,R)}
\label{local_strong_conv.h}
\end{align}
strongly under the convergence of Hilbert spaces $H_{i(j),b}' \to H_{\infty, b}'$, for any $R>0$, where 
$\hat{B}'(\infty,R)$ is the geodesic ball in 
$(C_m'\times S^1,g_{l,\infty})$ centered at 
the base point $(0,1)\in C_m'\times S^1$. 
To show that $f_{\infty,b}'$ is in $L^2$ (not just $L^2_{\mathrm{loc}}$), 
we use the fact that $\hat{B}'(s,u_b,R)\cap \hat{B}'(s,u_{b'},R) =\emptyset$ 
for sufficiently small $s>0$ if $b\neq b'$, 
which is shown by 
taking $\delta>0$ sufficiently small. 

Now let $\phi_{i,b}=F_{s_i}$ be the $S^1$-equivariant approximations defined in 
the proof of Proposition \ref{prop_S^1-pmGH.h} 
which gives the convergence 
\begin{align*}
(S,\hat{g}_{s_i},s_i^{-n/2}\nu_{\hat{g}_{s_i}},u_b)
&\stackrel{S^1\mathchar`-{\rm pmGH}}{\longrightarrow} (C_m\times S^1,g_{l,\infty},\det (A(0))^{-1/2}dtd\xi,(0,1)),
\end{align*}
and denote the $S^1$-equivariant approximations
$\phi_{i,b}'$, constructed in the same way,  
which gives the convergence 
\begin{align*}
((\mu_P\circ\pi^{-1}(B(b,\delta))),\hat{g}_{s_i,b}',\nu_{s_i,b}',u_b)
&\stackrel{S^1\mathchar`-{\rm pmGH}}{\longrightarrow} (C_m'\times S^1,g_{l,\infty},dtd\xi,(0,1)) 
\end{align*}
for some neighborhood 
$B(b,\delta)\subset \R^n$ of $b$.
By the strong convergence \eqref{local_strong_conv.h} and Definition \ref{def_conv_vector.y} (1), 
there is a sequence $\{\tilde{f}_{l,b,R}'\}_{l \in \N}\subset\mathcal{C}_b'$ 
such that 
$\lim_{l\to\infty}\tilde{f}_{l,b,R}' 
= f'_{\infty,b}|_{\hat{B}'(\infty,R)}$ and 
\begin{align*}
\lim_{l\to\infty}
\limsup_{j\to\infty} \left\| \tilde{f}_{l,b,R}'\circ 
\phi_{i(j),b}'
- f_{i(j)}|_{\hat{B}'(s_{i(j)},u_b,R)} \right\|_{H_{i(j),b}'}=0.
\end{align*}
Note that the above $L^2$-norm 
is the integral on 
$\hat{B}'(s_{i(j)},u_b,R)$, 
and we may take 
$\tilde{f}_{l,b,R}'$ such that 
${\rm supp}(\tilde{f}_{l,b,R}')
\subset \hat{B}'(\infty,R)$. 
Take $l_R>0$ such that 
\begin{align*}
\| \tilde{f}_{l,b,R}' - f'_{\infty,b}|_{\hat{B}'(\infty,R)}
\|_{H_\infty'}&\le 2^{-R},\\
\limsup_{j\to\infty} \left\| \tilde{f}_{l,b,R}'\circ 
\phi_{i(j),b}'
- f_{i(j)}|_{\hat{B}'(s_{i(j)},u_b,R)} \right\|_{H_{i(j),b}'}
&\le 2^{-R}
\end{align*}
for all $l\ge l_R$ and $b\in B_k$. 

Define $\Psi\colon C_m\times S^1\to C_m'\times S^1$ by 
\begin{align*}
\Psi(\xi,e^{\sqrt{-1}t}):=\left( A(0)^{-1/2} \cdot \xi ,\, e^{\sqrt{-1}t}\right), 
\end{align*}
and put 
$f_{\infty,b}:=f_{\infty,b}'\circ\Psi$, 
$\tilde{f}_{l,b,R}:=\tilde{f}_{l,b,R}'\circ\Psi\in \mathcal{C}_b$. 
From now on, we show that we have $f_{i(j)} \to (f_{\infty, b})_b$ strongly as $j \to \infty$, thus obtaining the result. 
In order to show this, by definition of strong convergence of vectors given in Definition \ref{def_conv_vector.y} (1), it is enough to check the following two conditions. 
\begin{enumerate}
    \item[(a)] We have $\tilde{f}_{l_R,b,R} \to f_{\infty,b}$ as $R \to \infty$ in $H_{\infty, b}$ for each $b \in B_k$. 
    \item[(b)] We have
    \begin{align*}
\lim_{R\to \infty}
\limsup_{j\to \infty}\left\| \sum_b\tilde{f}_{l_R,b,R}\circ\phi_{i(j),b}- f_{i(j)}\right\|_{L^2}
=0. 
\end{align*}
\end{enumerate}
First we show (a). We have
\begin{align*}
\| \tilde{f}_{l_R,b,R} - f_{\infty,b}\|_{L^2}^2
&= \| \tilde{f}_{l_R,b,R} 
- f_{\infty,b}|_{\Psi^{-1}(\hat{B}'(\infty,R))}\|_{L^2}^2 
+\| f_{\infty,b}|_{\Psi^{-1}(\hat{B}'(\infty,R))^c}\|_{L^2}^2
\\
&= \det A(0)\cdot\| \tilde{f}'_{l_R,b,R} 
- f_{\infty,b}'|_{\hat{B}'(\infty,R)}\|_{L^2}^2 
+\| f_{\infty,b}|_{\Psi^{-1}(\hat{B}'(\infty,R))^c}\|_{L^2}^2\\
&\le 2^{-2R} \det A(0)
+\| f_{\infty,b}|_{\Psi^{-1}(\hat{B}'(\infty,R))^c}\|_{L^2}^2\to 0
\end{align*}
as $R\to 0$, so we get the condition (a). 
Next we show (b). 
Since we have 
$\Psi\circ\phi_i=\phi_i'$ 
and $s^{-n/2}\nu_{\hat{g}_s}\le C\nu_{s,b}'$ for some 
$C>0$,
we have
\begin{align*}
&\quad
\limsup_{j\to \infty}\left\|
\sum_b\tilde{f}_{l_R,b,R}\circ\phi_{i(j),b}- f_{i(j)}\right\|_{H_{i(j)}}^2\\
&= 
\limsup_{j\to \infty}\int_S 
\bigg| \sum_b\tilde{f}_{l_R,b,R}\circ\phi_{i(j),b}- f_{i(j)}
\bigg|^2
s_{i(j)}^{-n/2}d\nu_{\hat{g}_{s_{i(j)}}}\\
&= 
\limsup_{j\to \infty}\sum_b
\int_{\hat{B}'(s_{i(j)},u_b,R)} 
\left| \tilde{f}_{l_R,b,R}'\circ\phi_{i(j),b}'- f_{i(j)}
\right|^2
s_{i(j)}^{-n/2}d\nu_{\hat{g}_{s_{i(j)}}} \\
&\quad + 
\limsup_{j\to\infty}
\int_{\bigcap_b 
\hat{B}'(s_{i(j)},u_b,R)^c} 
\left|  f_{i(j)} \right|^2
s_{i(j)}^{-n/2}d\nu_{\hat{g}_{s_{i(j)}}}\\
&\le 
\limsup_{j\to \infty}C\sum_b
\int_{\hat{B}'(s_{i(j)},u_b,R)} 
\left| \tilde{f}_{l_R,b,R}'\circ\phi_{i(j),b}'- f_{i(j)}
\right|^2
d\nu_{s_{i(j)},b}' \\
&\quad + 
\limsup_{j\to\infty}
\int_{\bigcap_b 
\hat{B}'(s_{i(j)},u_b,R)^c} 
\left|  f_{i(j)} \right|^2
s_{i(j)}^{-n/2}d\nu_{\hat{g}_{s_{i(j)}}}\\
&= 
\limsup_{j\to \infty}
\left(C\sum_b
\left\| \tilde{f}_{l_R,b,R}'\circ\phi_{i(j),b}'- f_{i(j)}|_{\hat{B}'(s_{i(j)},u_b,R)}
\right\|_{H_{i(j),b}'}^2
+ 
\left\|  
f_{i(j)}|_{\bigcap_b 
\hat{B}'(s_{i(j)},u_b,R)^c}
\right\|_{L^2}^2\right)\\
&\le
\limsup_{j\to \infty} \left( C\# B_k
\cdot 2^{-2R}
+ 
\| f_{i(j)}|_{\bigcap_b\hat{B}'(s_{i(j)},u_b,R)^c}\|_{L^2}^2\right).
\end{align*}
Since $({\rm ii})$ of 
Proposition \ref{prop_ball.h} and 
Remark \ref{rem_ball_submersion.h} gives 
\begin{align*}
\bigcap_b\hat{B}'(s,u_b,R)^c
\subset 
\left\{(\mu_P\circ\pi)^{-1}
(B(B_k, s(R-2\pi)))
\right\}^c,
\end{align*}
then Proposition \ref{thm_loc_bs_toric. y} gives 
\begin{align*}
\| f_i|_{\bigcap_b\hat{B}'(s_i,u_b,R)^c}\|_{L^2}
\to 0
\end{align*}
as $i\to \infty$. Hence we obtain  
\begin{align*}
\lim_{R\to \infty}
\limsup_{j\to \infty}\left\| \sum_b\tilde{f}_{l_R,b,R}\circ\phi_{i(j),b}- f_{i(j)}\right\|_{L^2}
=0,
\end{align*}
so we get (b). 
Thus $\{ f_{i(j)}\}_j$ 
strongly converges to 
$(f_{\infty, b})_b$. 
Since all of $f_i$ are 
$S^1$-equivariant, 
the limit $(f_{\infty, b})_b$ is also 
$S^1$-equivariant. 
\end{proof}
\begin{proof}[Proof of  
Theorem \ref{thm_cpt_conv.h}]
Take $s_i >0$ such that 
$\lim_{i \to \infty}s_i = 0$. 
Let $\Sigma_i$ be the spectral 
structure given by 
$\Delta_{\hat{g}_{J_{s_i}}}^{\rho^k}$. 
It suffices to prove that $\Sigma_i \to \Sigma_\infty$ compactly. 
By Proposition \ref{prop_str_conv.y}, we see that $\Sigma_i \to \Sigma_\infty$ strongly. 
By Definition \ref{def_quad_conv.y} and Fact \ref{fact_KS_conv.h}, 
we need to show that, for any $\{ u_i\}_i$ with 
$\limsup_{i\to\infty}(\| u_i\|_{H_i}^2 
+ \| du_i\|_{H_i}^2) < \infty$, 
there exists a strongly convergent subsequence. 
If all of $u_i$ 
are smooth, then 
it is shown by Proposition 
\ref{prop_compact_conv.h}. 
In general for not necessarily smooth $\{u_i\}_{i}$, 
we can approximate $\{u_i\}_i$ by a sequence $\{u'_i\}_{i}$ with $u'_i \in (C^\infty(S)\otimes \C)^{\rho_k}$, $\lim_i \|u_i - u'_i\| = 0$ and $\limsup_{i\to\infty}(\| u'_i\|_{H_i}^2 + \| du'_i\|_{H_i}^2) < \infty$, so we get the result. 
\end{proof}

Restricting the above spectral convergence result to the zero-eigenspaces, we obtain the convergence result of quantum Hilbert spaces as follows.  
\begin{thm}\label{thm_conv_q_hilb.h}
Let $k$ be a positive integer. 
Let us denote the orthogonal projection on $L^2(X_P, g_s; L^k)$ to the subspace $H^0((X_P)_{J_s}; L^k)$ by $P_{k,s}$. 
Let us also consider the subspace $\ker \Delta_{\mathcal{C}_b(\psi)}^k \subset L^2(\mathcal{C}_b(\psi), e^{-k\|\xi\|^2}d\xi)\otimes \C$, which is one-dimensional by Proposition \ref{prop_spec_Gaussian.y}, and denote by $P_k^b$ the projection onto this subspace. 
Then, under the convergence of Hilbert spaces $L^2((X_P)_{J_s}; L^k) \to \bigoplus_{b \in B_k}L^2(\mathcal{C}_b(\psi), e^{-k\|\xi\|^2}d\xi)\otimes \C $ as $s \to 0$, we have a compact convergence
\begin{align*}
    P_{k, s} \xrightarrow{s \to 0} \bigoplus_{b \in B_k} P_k^b, 
\end{align*}
as a family of bounded operators on this family. 
\end{thm}
\begin{proof}
In this proof, 
we use the following well-known equality 
\begin{align}
    \dim H^0(X_{J_s}; L^k) = \# B_k 
    = \#(P\cap\Z^n)\label{eq_toric_index.h}
\end{align}
for a compact toric symplectic  manifold $X_P$. 
Denote by $\lambda_1>k^2+kn$ be 
the minimum of the eigenvalues 
of $(\Delta_\infty^b)^{\rho_k}$ 
larger than $k^2+kn$. 
Take $\delta>0$ such that 
$\lambda_1-\delta>k^2+kn$. 
By Theorem \ref{thm_cpt_conv.h}, 
$E_s((k^2+kn-1,\lambda_1-\delta])\to 
E_\infty((k^2+kn-1,\lambda_1-\delta])$ 
compactly. 
Moreover, by \cite[Theorem 2.6]{KuwaeShioya2003}, 
we have 
$\dim E_\infty((k^2+kn-1,\lambda_1-\delta])= \dim E_s((k^2+kn-1,\lambda_1-\delta])$ 
for sufficiently small $s>0$. 
Since we have 
\begin{align*}
\dim E_\infty((0,\lambda_1-\delta])
=\#B_k, 
\quad {\rm Ker}(\Delta_{\hat{g}_s})
\subset E_s((k^2+kn-1,\lambda_1-\delta])
\end{align*}
and \eqref{eq_toric_index.h}, 
then we have 
${\rm Ker}(\Delta_{\hat{g}_s})
= E_s((k^2+kn-1,\lambda_1-\delta])$ 
for sufficiently small $s>0$. 
\end{proof}
\begin{rem}
\normalfont
Theorem \ref{thm_conv_q_hilb.h} corresponds 
to \cite[Theorem 1.3]{BFMN2011}, 
however, there are some difference. 
The authors of \cite{BFMN2011} constructed 
a family of the basis of 
$H^0(X_{J_s}; L^k)$ concretely 
and show the convergence of them 
in the sense of distributions as $s\to 0$. 
In our case, although we can obtain a strongly 
converging family of the basis of 
$H^0(X_{J_s}; L^k)$ by the compact convergence 
of the projections $\{ P_{k,s}\}_s$, 
they are not described concretely.
\end{rem}

\section*{Acknowledgment}
The authors are grateful to Shouhei Honda for useful discussions. 
K.Hattori is supported by Grant-in-Aid for 
Scientific Research (C) 
Grant Number 19K03474.
M. Yamashita is supported by Grant-in-Aid for JSPS Fellows Grant Number 19J22404.

\bibliographystyle{plain}

\end{document}